\documentclass[a4paper,11pt,reqno]{amsart}

\usepackage{url}
\usepackage[colorlinks=true, linkcolor=blue, citecolor=ForestGreen]{hyperref}
\usepackage{cite}

\usepackage{latexsym,amsmath,amssymb,amsfonts,amsthm}
\usepackage{mathrsfs}
\usepackage{mathtools}
\usepackage{dsfont}
\usepackage{bbm}
\usepackage{soul}     
\usepackage[utf8]{inputenc}

\usepackage[usenames, dvipsnames]{color}
\usepackage[usenames, dvipsnames, cmyk]{xcolor}
\usepackage{graphicx}
\usepackage[scriptsize,hang,raggedright]{subfigure}
\usepackage{multirow}
\usepackage{enumerate}
\usepackage{enumitem}

\usepackage{a4wide}
\usepackage{epsfig,epstopdf}

\numberwithin{equation}{section}

\definecolor{grey}{rgb}{.7,.7,.7}
\definecolor{refkey}{gray}{.45}
\definecolor{labelkey}{gray}{.45}

\newcommand{\xupref}[2]{\hspace{-0.3ex}\stackrel{\eqref{#1}}{#2}} 

\newtheorem{theorem}{Theorem}[section]
\newtheorem{proposition}[theorem]{Proposition}
\newtheorem{lemma}[theorem]{Lemma}

\newtheorem{corollary}[theorem]{Corollary}

\theoremstyle{remark}
\newtheorem{remark}[theorem]{Remark}
\theoremstyle{definition}

\newtheorem{example}[theorem]{Example}

\usepackage{tikz}
\usepackage{pgf,pgfplots}
\usetikzlibrary{calc}
\usetikzlibrary{decorations.markings}
\usetikzlibrary{decorations.pathmorphing}
\usetikzlibrary{decorations.shapes}
\usetikzlibrary{shapes,arrows,shapes.geometric,patterns,fadings}

\newcommand{\e}{\varepsilon}

\newcommand{\N}{\mathbb N}

\newcommand{\R}{\mathbb R}

\newcommand{\dd}{\,\mathrm{d}}
\renewcommand{\setminus}{\backslash}

\newcommand{\defeq}{\coloneqq}

\newcommand{\en}{\mathcal{F}}
\newcommand{\nl}{\mathcal{J}}
\newcommand{\nla}{\mathcal{J}_{-\alpha}}
\newcommand{\nlb}{\mathcal{J}_{\beta}}
\newcommand{\I}{\mathcal{I}}

\newcommand{\Hd}{\mathcal{H}^{d-1}}
\newcommand{\asymm}[2]{#1\triangle#2}
\newcommand{\Asymm}[1]{\Delta(#1)}

\newcommand{\ba}{\begin{array}}
\newcommand{\ea}{\end{array}}

\newcommand{\tld}[1]{\widetilde{#1}}
\newcommand{\bthm}{\begin{theorem}}
\newcommand{\ethm}{\end{theorem}}
\newcommand{\bprop}{\begin{proposition}}
\newcommand{\eprop}{\end{proposition}}
\newcommand{\blemma}{\begin{lemma}}
\newcommand{\elemma}{\end{lemma}}
\newcommand{\bexmpl}{\begin{example}}
\newcommand{\eexmpl}{\end{example}}

\newcommand{\beqn}{\begin{equation}}
\newcommand{\eeqn}{\end{equation}}
\newcommand{\beqns}{\begin{equation*}}
\newcommand{\eeqns}{\end{equation*}}

\newcommand{\pr}{\prime}

\renewcommand{\leq}{\leqslant}
\renewcommand{\geq}{\geqslant}

\definecolor{mygreen}{rgb}{0.1,0.75,0.2}

\newcounter{myenumi}
\DeclareMathOperator{\dive}{div}

\DeclareMathOperator{\diam}{diam}

\allowdisplaybreaks

\title[Stability of the ball for attractive - repulsive energies]{Stability of the ball for attractive - repulsive energies}

\author{Marco Bonacini}
\address[Marco Bonacini]{Department of Mathematics, University of Trento, Italy}
\email{marco.bonacini@unitn.it}

\author{Riccardo Cristoferi}
\address[Riccardo Cristoferi]{Department of Mathematics - IMAPP, Radboud University, Nijmegen, The Netherlands}
\email{riccardo.cristoferi@ru.nl}

\author{Ihsan Topaloglu}
\address[Ihsan Topaloglu]{Department of Mathematics and Applied Mathematics, Virginia Commonwealth University, Richmond, VA, USA}
\email{iatopaloglu@vcu.edu}

\date{\today}       

\thanks{This is a post-peer-review, pre-copyedit version of an article published in the SIAM Journal on Mathematical Analysis.
The final authenticated version is available online at: \url{https://doi.org/10.1137/22M1506894}.}                                 


\begin{document}

\begin{abstract}
We consider a class of attractive-repulsive energies, given by the sum of two nonlocal interactions with power-law kernels, defined over sets with fixed measure. It has recently been proved by R. Frank and E. Lieb that the ball is the unique (up to translation) global minimizer for sufficiently large mass. We focus on the issue of the stability of the ball, in the sense of the positivity of the second variation of the energy with respect to smooth perturbations of the boundary of the ball. We characterize the range of masses for which the second variation is positive definite (large masses) or negative definite (small masses). Moreover, we prove that the stability of the ball implies its local minimality among sets sufficiently close in the Hausdorff distance, but not in $L^1$-sense.
\end{abstract}

\maketitle


\section{Introduction}\label{sec:intro}

In \cite{BurChoTop18} the following minimization problem among measurable sets $E\subset\R^d$ with fixed volume $|E|=m$ is considered:
\begin{equation} \label{eq:min}
\min\Bigl\{ \en(E) \;:\; E\subset\R^d,\, |E|=m \Bigr\},
\end{equation}
where the functional $\en$ is a prototypical attractive-repulsive energy with power-law kernels
\begin{equation}\label{eq:energy}
\en(E)\defeq \int_E\int_E \frac{1}{|x-y|^\alpha}\dd x\dd y + \int_E\int_E |x-y|^\beta\dd x\dd y, \qquad \alpha\in(0,d),\,\beta>0.
\end{equation}
Such nonlocal interaction energies arise in descriptions of systems of uniformly distributed interacting particles, in particular in models of collective behavior of many-agent systems related to swarming (see e.g. \cite{BernoffTopaz,FetecauHuangKolokolnikov,TopazBertozzi2}).
For an introductory survey on this and related problems, see \cite{Fra}.

As a simple scaling argument shows, the first term in \eqref{eq:energy} is a repulsive interaction and is dominant in the small mass regime, whereas the second term is attractive and is dominant in the large mass regime. It is convenient to introduce a notation for the kernels and for the attractive and repulsive parts of the energy: for $\sigma>-d$ we set
\begin{equation}\label{eq:nl_interaction}
K_\sigma(r) \defeq r^\sigma, \qquad \nl_\sigma(E)\defeq \int_E\int_E |x-y|^\sigma\dd x\dd y.
\end{equation}
With this notation $\en(E)=\nl_{-\alpha}(E)+\nl_{\beta}(E)$.

Several results appeared recently in the mathematical literature about the minimization of energies of the form \eqref{eq:energy} (also for more general kernels) and about the corresponding relaxed problem among uniformly bounded densities
\begin{equation} \label{eq:minrelaxed}
\begin{split}
\min\biggl\{ \en(\rho)\defeq \int_{\R^d}\int_{\R^d} \frac{\rho(x)\rho(y)}{|x-y|^\alpha}&\dd x\dd y + \int_{\R^d}\int_{\R^d} \rho(x)\rho(y)|x-y|^\beta\dd x\dd y \;:\; \\
& \rho\in L^1(\R^d)\cap L^\infty(\R^d),\, 0\leq\rho\leq1,\, \int_{\R^d}\rho(x)\dd x =m \biggr\}.
\end{split}
\end{equation}
Existence of minimizers for \eqref{eq:minrelaxed} is proved in \cite{ChFeTo14}, whereas in \cite{BurChoTop18} it is shown that a set $E$ is a minimizer for \eqref{eq:min} if and only if its characteristic function $\rho=\chi_E$ solves the relaxed problem \eqref{eq:minrelaxed}. While the minimization problem \eqref{eq:minrelaxed} has also been studied over probability measures, in \cite{CDM14} it is shown that for $d-2\leq \alpha < d$ and $\beta>0$ the optimal measures are actually $L^\infty$-functions. Recently, in \cite{CarPra}, this result is extended to a wide class of interaction kernels, where the authors also provide an a priori bound on the $L^\infty$-norm of minimizers. The problem \eqref{eq:minrelaxed} is in particular relevant to models of crowd motion as investigated in \cite{MaRoSa10,MaRoSaVe11} where the density $\rho$ describing the distribution of individuals is uniformly bounded. A question of interest is when the optimal densities are uniformly distributed on their support and when the constraint is saturated. In this case the problem \eqref{eq:minrelaxed} reduces to \eqref{eq:min} where $\rho$ is described by the characteristic function of a set.

It is known that, for certain ranges of the parameters $\alpha$, $\beta$, \textit{balls are global minimizers of \eqref{eq:energy} for large mass}: more precisely, by using quantitative rearrangement inequalities Frank and Lieb in \cite{FraLie21} proved that for all $\beta>0$ and $0<\alpha<d-1$ there is a threshold $m_{\rm ball}\in(0,\infty)$ such that the ball with mass $m$ is the only (up to translation) minimizer of $\en$ for $m>m_{\rm ball}$.
This result had already been obtained in the special case of quadratic attraction ($\beta=2$) by Burchard, Choksi and the third author in \cite{BurChoTop18}, by exploiting the convexity of the energy among densities $\rho$ with zero center of mass; a similar argument was afterwards used by Lopes in \cite{Lop19} to prove that, for $2<\beta\leq 4$ (and any $\alpha<d$) minimizers of \eqref{eq:minrelaxed} are unique up to translation and radially symmetric.
It is worth to remark that, as observed in \cite{FraLie21}, for $\alpha\in[d-1,d)$ \textit{characteristic functions of balls are never critical points for the relaxed energy} and therefore they are never minimizers (see also Remark~\ref{rmk:ball-noncrit}).

Concerning the small mass regime, in \cite{BurChoTop18} it is shown that for $\beta=2$ and $\alpha\in[d-2,d)$ the minimization problem \eqref{eq:min} does not have a solution for all masses $m$ sufficiently small: indeed, in this case the minimizer of \eqref{eq:minrelaxed}, which is unique by convexity of the energy, satisfies $\rho(x)<1$ almost everywhere. In \cite{FraLie18} the same is proved for $d=3$, $\alpha=1$, and any $\beta>0$.

Finally, in the Coulomb case ($\alpha=d-2$) with quadratic attraction ($\beta=2$) the minimization problem \eqref{eq:min} is completely solved in \cite{BurChoTop18}: the unique solution is the ball if $m\geq m_{\rm ball}$, while there is no solution if $m<m_{\rm ball}$, and $m_{\rm ball}=\frac{(d-2)}{2}\omega_d$, where $\omega_d\defeq |B_1|$ denotes the volume of the ball in $\R^d$ with radius 1. The explicit value of $m_{\rm ball}$ is also computed in \cite{Lop19} for the cases $d=3$, $\alpha=1$, and $\beta=3,4$.

\medskip

In this paper we address the issues of the \emph{stability and local optimality of the ball}. In our main result we characterize the range of masses for which the ball is a volume-constrained stable set for $\en$ and for $-\en$ (in the sense of nonnegativity of the second variation of the energy with respect to smooth perturbations of the boundary of the ball, see Section~\ref{subsect:var} for the precise definition). Moreover, we prove that the strict stability of the ball yields its (quantitative) local minimality/maximality among sets with the same volume and whose boundary is contained in a small uniform neighbourhood of the boundary of the ball.
The mass thresholds which separate the volume-constrained stability/instability regions for the ball are determined explicitly in terms of the parameters $\alpha\in(0,d-1)$, $\beta>0$, and are given by the following constants:
\begin{equation} \label{eq:m*}
m_* = m_*(d,\alpha,\beta) \defeq
\begin{cases}\displaystyle
\displaystyle \omega_d\biggl( \frac{\alpha(d-\alpha)(2d+2+\beta)}{\beta(d+\beta)(2d+2-\alpha)}\cdot\frac{\nla(B_1)}{\nlb(B_1)} \biggr)^{\frac{d}{\alpha+\beta}} & \text{if }\beta\geq \beta_*, \\[3ex]
\displaystyle \omega_d \biggl( \frac{\alpha(d-\alpha)(2d-\alpha)(d-1+\beta)}{\beta(d+\beta)(2d+\beta)(d-1-\alpha)}\cdot\frac{\nla(B_1)}{\nlb(B_1)} \biggr)^{\frac{d}{\alpha+\beta}}& \text{if }\beta<\beta_*,
\end{cases}
\end{equation}
where
\begin{equation} \label{eq:beta*}
\beta_*=\beta_*(d,\alpha)\defeq \frac{6d+2+\alpha(d-1)}{d-1-\alpha}\,,
\end{equation}
and
\begin{equation} \label{eq:m**}
m_{**} = m_{**}(d,\alpha,\beta) \defeq \omega_d \biggl( \frac{\alpha(d-\alpha)}{\beta(d+\beta)}\cdot \frac{\nla(B_1)}{\nlb(B_1)} \biggr)^{\frac{d}{\alpha+\beta}}.
\end{equation}
The values of $\nla(B_1)$ and $\nlb(B_1)$ which appear in the previous formulas can be expressed in terms of the $\Gamma$ function (see Remark~\ref{rmk:explicitGamma}), so that $m_*$ and $m_{**}$ can be computed explicitly. For every $m>0$ we also denote by $r(m)>0$ the radius of a ball with volume $m$, $|B_{r(m)}|=m$. We have the following theorem, which is the main result of the paper.

\begin{theorem} \label{thm:main}
Let $\alpha\in(0,d-1)$, $\beta>0$, and let $m_*$, $m_{**}$ be defined by \eqref{eq:m*} and \eqref{eq:m**} respectively. Then:
\begin{enumerate}
\item the ball $B_{r(m)}$ with mass $m$ is a volume-constrained stable set for the functional $\en$ if and only if $m\geq m_*$,
\item the ball $B_{r(m)}$ with mass $m$ is a volume-constrained stable set for the functional $-\en$ if and only if $m\leq m_{**}$.
\end{enumerate}
Moreover, for every $m\in(0,m_{**})\cup(m_*,\infty)$ there exist $\overline{C}>0$ and $\bar\e>0$ (depending only on $d$, $\alpha$, $\beta$, and $m$) with the following property: for every measurable set $E\subset\R^d$ with $|E|=|m|$ and such that 
\begin{equation} \label{eq:unif-ball}
B_{r(m)(1-\bar{\e})}(x_0) \subset E \subset B_{r(m)(1+\bar{\e})}(x_0)
\end{equation}
for some $x_0\in\R^d$, one has
\begin{align}
\en(E) - \en(B_{r(m)}) \geq \overline{C} (\Asymm{E})^2 & \qquad\text{if }m>m_{*}, \label{eq:quant1}\\
\en(E) - \en(B_{r(m)}) \leq - \overline{C} (\Asymm{E})^2 & \qquad\text{if }m<m_{**}, \label{eq:quant2} 
\end{align}
where $\Asymm{\cdot}$ denotes the asymmetry
\begin{equation} \label{eq:asymm}
\Asymm{E} \defeq \inf_{y\in\R^d}|\asymm{E}{B_{r(|E|)}(y)}|.
\end{equation}
\end{theorem}

The thresholds $m_*$ and $m_{**}$ are determined by means of a spherical harmonics expansion of the quadratic form associated to the second variation of the energy (for stability analyses using similar techniques see \cite{BC14,FFMMM,FraLie19_note,FusPra20,Asc22}). The second part of the statement, namely the local optimality of the ball, is proved in two steps. In the first one (see Section~\ref{sec:fuglede}) we show by a standard Fuglede-type argument that the stability inequalities \eqref{eq:quant1}--\eqref{eq:quant2} are valid whenever $E$ is a nearly-spherical set, that is, the boundary of $E$ is a normal graph over the boundary of the ball and is uniformly close to it. The second step amounts to improve the local optimality to the larger class of competitors whose boundary is contained in a small tubular neighbourhood of the boundary of the ball, in the sense of \eqref{eq:unif-ball}: in this case we exploit a transport argument devised by Fusco and Pratelli in \cite{FusPra20} which allows to replace any such set by a nearly-spherical set and hence to exploit the previous step of the proof. This strategy is similar in spirit to the argument that Christ used in \cite{Christ17} (see also \cite{FraLie19_note,FraLie21}).

It would be natural to see whether or not the previous result can be further extended to an $L^1$-neighbourhood of the ball, namely whether \eqref{eq:quant1} holds for every set $E$ such that $\Asymm{E}$ is small enough. This is indeed usually the case in similar problems where the attractive part of the energy is given by the perimeter. However, in our case the lack of the regularizing effect of the perimeter prevents such a possibility. In Section~\ref{sec:L1minimality} we discuss a necessary condition for the ball to be an $L^1$-local minimizer; in particular, we show that there are values of the parameters such that this condition is not satisfied for some $m>m_*$, so that the result cannot actually be improved.

Notice that the same question for \eqref{eq:quant2} has obviously a negative answer, since the energy can be made arbitrarily large in any $L^1$-neighbourhood of the ball. Indeed, for every $r>0$ the family of sets $E_k^{(r)}\coloneqq (B_1\setminus B_r)\cup B_r(ke_1)$, $k\in\N$, obtained by removing a little ball from $B_1$ and sending it to infinity, are such that $\en(E_k^{(r)})\to\infty$ as $k\to\infty$, whereas $|\asymm{E_k^{(r)}}{B_1}|\leq 2|B_r|$ can be made arbitrarily small.

\begin{remark} \label{rmk:explicit}
In the case $\alpha=d-2$, $\beta=2$, the stability threshold $m_*$ coincides with the global minimality threshold $m_{\rm ball}=\frac{(d-2)}{2}\omega_d$ computed in \cite{BurChoTop18}. This fact is not surprising, since for quadratic attraction the functional is strictly convex among densities with zero center of mass. In the other cases in which the global minimality threshold is known explicitly ($d=3$, $\alpha=1$ and $\beta=3,4$, see \cite{Lop19}) one has instead the strict inequality $m_*<m_{\rm ball}$: there is therefore an intermediate regime in which the ball is a local minimizer among perturbations as those considered in Theorem~\ref{thm:main}, but ceases to be a global minimizer. In these cases, for $m_*<m<m_{\rm ball}$ the potential of the ball does not satisfy the necessary condition for the $L^1$-local minimality discussed in Section~\ref{sec:L1minimality}.
\end{remark}

In the rest of the paper, it will be convenient to normalize the volume constraint and to work with sets having the fixed mass $|E|=|B_1|$. By scaling (see the proof of Theorem~\ref{thm:main} at the end of Section~\ref{sec:locmin}) this amounts to introduce a parameter in the functional $\en$: we will consider for $\gamma>0$ the energy
\begin{equation}\label{eq:energy-gamma}
\en_\gamma(E)\defeq \int_E\int_E \frac{1}{|x-y|^\alpha}\dd x\dd y + \gamma \int_E\int_E |x-y|^\beta\dd x\dd y
\end{equation}
so that in particular $\en=\en_1$. 

\medskip\noindent\textbf{Structure of the paper.}
In Section~\ref{sec:stability} we compute the second variation of the energy and we determine its sign by means of a spherical harmonics expansion. We then show the local minimality/maximality of the ball among nearly-spherical sets in Section~\ref{sec:fuglede}, and among more general perturbations which are close in the Hausdorff distance in Section~\ref{sec:locmin}, where the proof of Theorem~\ref{thm:main} is completed. In Section~\ref{sec:L1minimality} we discuss a necessary condition for the $L^1$-local minimality of the ball in terms of its potential.


\section{Stability of the ball}\label{sec:stability}

In this section we discuss the volume-constrained stability of the unit ball for the family of functionals $\en_\gamma$, $\gamma>0$, see \eqref{eq:energy-gamma}. In the main result of this section, Theorem~\ref{thm:stability}, we obtain two explicit thresholds $0<\gamma_{**}<\gamma_*<\infty$ separating the stability/instability regions for $\en_\gamma$ and $-\en_\gamma$.


\subsection{First and second variations} \label{subsect:var}
We start by recalling the notions of volume-constrained stationary set and stable set for the energy $\en_\gamma$. Given a vector field $X\in C^\infty_{\mathrm c}(\R^d;\R^d)$, the \emph{flow $\{\Phi_t\}_t$ induced by $X$} is the map $\Phi:(t,x)\in\R\times\R^d\mapsto \Phi_t(x)$ implicitly defined by the system of ODE's
\begin{equation} \label{eq:flow}
\begin{cases}
\partial_t\Phi(t,x)= X(\Phi(t,x)), \\
\Phi(0,x)=x,
\end{cases}
\end{equation}
which is a smooth family of diffeomorphisms for $|t|<\e$, for some $\e>0$. Given a set $E\subset\R^d$ with finite volume, we say that $X$ induces a \emph{volume-preserving flow on $E$} if $|\Phi_t(E)|=|E|$ for every $t\in(-\e,\e)$.

Given a set $E$ with finite volume and a vector field $X\in C^\infty_{\mathrm c}(\R^d;\R^d)$, we define the \emph{first} and \emph{second variations} of $\en_\gamma$ at $E$, along the flow $\{\Phi_t\}_t$ induced by $X$, as
\begin{equation}\label{eq:variations}
\delta\en_\gamma(E)[X] \defeq \frac{\dd}{\dd t}\en_\gamma(\Phi_t(E))\Big|_{t=0},
\qquad
\delta^2\en_\gamma(E)[X] \defeq \frac{\dd^2}{\dd t^2}\en_\gamma(\Phi_t(E))\Big|_{t=0}.
\end{equation}
We then say that the set $E$ is a \emph{volume-constrained stationary set} for $\en_\gamma$ if $\delta\en_\gamma(E)[X]=0$ for every vector field $X$ which induces a volume-preserving flow on $E$, and that $E$ is a \emph{volume-constrained stable set} for $\en_\gamma$ if it is stationary and in addition $\delta^2\en_\gamma(E)[X]\geq0$ for every vector field $X$ which induces a volume-preserving flow on $E$.

Analogous definitions can be of course given for the nonlocal interaction $\nl_\sigma$ alone, see \eqref{eq:nl_interaction}. The following theorem provides explicit formulas for the first and second variation of $\nl_\sigma$, for $\sigma\in(-d,\infty)$. The result is proved, for instance, in \cite[Theorem~6.1]{FFMMM} in the case $\sigma\in(-d,0)$, but it is straightforward to check that the proof extends to $\sigma\geq0$.

\begin{theorem}[First and second variation of $\nl_\sigma$] \label{thm:var}
Let $\sigma\in(-d,\infty)$, let $E\subset\R^d$ be an open and bounded set with $C^2$ boundary, and let $X\in C^\infty_{\mathrm c}(\R^d;\R^d)$. Then, denoting by $\zeta=X\cdot\nu_E$ the normal component of the vector field $X$ on $\partial E$, one has
\begin{align}
\delta\nl_\sigma(E)[X] & = \int_{\partial E} H_{\sigma,\partial E}(x)\zeta(x)\dd\Hd_x, \label{eq:Ivar} \\
\delta^2\nl_\sigma(E)[X] & = - \int_{\partial E}\int_{\partial E} K_\sigma(|x-y|) |\zeta(x)-\zeta(y)|^2\dd\Hd_x\dd\Hd_y + \int_{\partial E} c^2_{\sigma,\partial E}\zeta^2\dd\Hd \nonumber \\
& \qquad + \int_{\partial E} H_{\sigma,\partial E}\Bigl( (\dive X)\zeta - \dive_\tau(\zeta X_\tau) \Bigr)\dd\Hd, \label{eq:IIvar}
\end{align}
where $\dive_\tau$ is the tangential divergence on $\partial E$, $X_\tau\defeq X-\zeta\nu_E$, and for $x\in\partial E$ we define
\begin{equation} \label{eq:nlcurvature}
H_{\sigma,\partial E}(x) \defeq 2\int_E K_\sigma(|x-y|)\dd y,
\quad
c^2_{\sigma,\partial E}(x)\defeq \int_{\partial E}K_\sigma(|x-y|)|\nu_E(x)-\nu_E(y)|^2\dd\Hd_y.
\end{equation}
\end{theorem}

\begin{remark}\label{rmk:var}
Notice that, if $E$ is a volume-constrained stationary set for $\nl_\sigma$ and $X$ induces a volume-preserving flow on $E$, then the last integral in the formula \eqref{eq:IIvar} of the second variation vanishes (see \cite[Remark~6.2]{FFMMM}).
\end{remark}

Obviously the unit ball $B_1$ is a volume-constrained stationary set for $\en_\gamma$, since $H_{\sigma,\partial B_1}$ is constant on $\partial B_1$ for all $\sigma$, and if $X$ induces a volume-preserving flow on $B_1$ then $\int_{\partial B_1}\zeta\dd\Hd=0$.
With the purpose of studying the stability of $B_1$, and in view of \eqref{eq:IIvar}, we introduce the quadratic form
\begin{equation} \label{eq:quad_form}
\begin{split}
\mathcal{Q}\en_\gamma & (u) \defeq
- \int_{\partial B_1}\int_{\partial B_1} \frac{|u(x)-u(y)|^2}{|x-y|^\alpha}\dd\Hd_x\dd\Hd_y + c^2_{-\alpha,\partial B_1}\int_{\partial B_1} u^2\dd\Hd \\
& - \gamma \int_{\partial B_1}\int_{\partial B_1} |x-y|^\beta |u(x)-u(y)|^2\dd\Hd_x\dd\Hd_y + \gamma c^2_{\beta,\partial B_1} \int_{\partial B_1} u^2\dd\Hd,
\end{split}
\end{equation} 
so that if $X\in C^\infty_{\mathrm c}(\R^d;\R^d)$ induces a volume-preserving flow, then by Theorem~\ref{thm:var} and Remark~\ref{rmk:var} the second variation of $\en_\gamma$ at $B_1$ with respect to $X$ is
\begin{equation}\label{eq:quadform-2}
\delta^2\en_\gamma(B_1)[X] = \mathcal{Q}\en_\gamma (X\cdot\nu_{B_1}).
\end{equation}

\begin{remark}
The following expression for the constant $c^2_{\sigma,\partial B_1}$ will be useful later on:
\begin{equation} \label{eq:nlcurvature_ball}
c^2_{\sigma,\partial B_1}(x) = \frac{(d+\sigma)(2d+\sigma)}{d\omega_d}\,\nl_\sigma(B_1), \qquad\text{for all } x\in\partial B_1.
\end{equation}
To prove \eqref{eq:nlcurvature_ball}, we observe that by using the divergence theorem we have for $x\in\partial B_1$
\begin{align*}
c^2_{\sigma,\partial B_1}(x)
& = \frac{1}{d\omega_d}\int_{\partial B_1}\int_{\partial B_1}K_\sigma(|x-y|)|x-y|^2\dd\Hd_x\dd\Hd_y \\
& = \frac{2}{d\omega_d}\int_{\partial B_1}\int_{\partial B_1}K_\sigma(|x-y|)(x-y)\cdot x \dd\Hd_x\dd\Hd_y \\
& = \frac{2}{d\omega_d}\int_{\partial B_1}\int_{B_1} \dive_{x}\Bigl( K_\sigma(|x-y|)(x-y)\Bigr) \dd x\dd\Hd_y \\
& = \frac{2(\sigma+d)}{d\omega_d}\int_{\partial B_1}\int_{B_1} K_\sigma(|x-y|) \dd x\dd\Hd_y \\
& = \frac{\sigma+d}{d\omega_d} \frac{\dd}{\dd r}\nl_\sigma(B_r)\Big|_{r=1}
= \frac{\sigma+d}{d\omega_d} \frac{\dd}{\dd r}\Bigl(r^{2d+\sigma}\nl_\sigma(B_1)\Bigr)\Big|_{r=1}
= \frac{(d+\sigma)(2d+\sigma)}{d\omega_d}\,\nl_\sigma(B_1).
\end{align*}
\end{remark}


\subsection{Decomposition in spherical harmonics} \label{subsect:spherharm}
The quadratic form $\mathcal{Q}\en_\gamma(u)$ can be conveniently expressed in terms of the Fourier decomposition of $u$ with respect to the orthonormal basis of spherical harmonics. We denote by $\mathcal{S}_k$ the finite dimensional subspace of spherical harmonics of degree $k\in\N\cup\{0\}$, and by $\{Y_k^i\}_{i=1}^{d(k)}$ an orthonormal basis (of dimension $d(k)$) for $\mathcal{S}_k$ in $L^2(\partial B_1)$ (see for instance \cite{Gro}). For $u\in L^2(\partial B_1)$, we let
\begin{equation} \label{eq:fourier}
a^i_k(u)\defeq \int_{\partial B_1} u \, Y^i_k \dd\Hd
\end{equation}
be the Fourier coefficient of $u$ corresponding to $Y^i_k$. We have
\begin{equation} \label{eq:L2norm}
\|u\|_{L^2(\partial B_1)}^2 = \sum_{k=0}^\infty\sum_{i=1}^{d(k)} \bigl(a^i_k(u)\bigr)^2.
\end{equation}
We can also express the seminorm $\int_{\partial B_1}\int_{\partial B_1} K_\sigma(|x-y|)|u(x)-u(y)|^2\dd\Hd_x\dd\Hd_y$ in terms of the Fourier coefficients of $u$, as shown in the following proposition.

\begin{proposition}\label{prop:seminorm}
For $\sigma>-(d-1)$ we have
\begin{equation} \label{eq:seminorm}
\int_{\partial B_1}\int_{\partial B_1} K_\sigma(|x-y|)|u(x)-u(y)|^2\dd\Hd_x\dd\Hd_y = \sum_{k=0}^\infty\sum_{i=1}^{d(k)} \mu_k(\sigma) \bigl(a^i_k(u)\bigr)^2,
\end{equation}
 where the sequence $(\mu_k(\sigma))_k$, for $\sigma\in(-(d-1),\infty)$, is defined as
\begin{equation} \label{eq:mu}
\mu_{k}(\sigma) \defeq (d-1)\omega_{d-1}2^{d-1+\sigma} \, \frac{\Gamma\left(\frac{d-1+\sigma}{2}\right)\Gamma\left(\frac{d-1}{2}\right)}{\Gamma\left(\frac{2d-2+\sigma}{2}\right)}\Biggl[1-\prod_{j=0}^{k-1}\frac{j-\frac{\sigma}{2}}{j+d-1+\frac{\sigma}{2}}\Biggr]
\end{equation}
for $k\geq1$, and $\mu_0(\sigma)=0$.
\end{proposition}

\begin{proof}
The formula is derived in \cite[(7.5) and (7.12)]{FFMMM} in the case $\sigma\in(-(d-1),0)$, and in \cite[Lemma~4.3, Corollary~3.6, and Lemma~3.7]{Asc22} in the case $\sigma>0$.
\end{proof}

\begin{remark} \label{rmk:alphalarge}
Also in the case $\sigma\in(-d,-(d-1)]$ a representation formula as in Proposition~\ref{prop:seminorm} holds true, but in this case the expression of the coefficients $\mu_k(\sigma)$ is different: see \cite[equations~(7.4), (7.6)]{FFMMM}.
\end{remark}

 We collect in the following lemma some useful properties of the sequence $(\mu_k(\sigma))_k$, defined by \eqref{eq:mu}.

\begin{lemma} [Properties of $(\mu_k(\sigma))_k$] \label{lemma:muk}
Let $\sigma>-(d-1)$. Then:
\begin{enumerate}
\item If $\sigma<0$, then $\mu_{k+1}(\sigma)>\mu_k(\sigma)$ for all $k$, and in particular
\begin{equation} \label{eq:muk-1}
\mu_k(\sigma) - \mu_1(\sigma) \geq \mu_2(\sigma)-\mu_1(\sigma) = -\frac{\sigma \mu_1(\sigma)}{2d+\sigma}>0 \qquad\text{for all }k\geq2.
\end{equation}
\item If $\sigma>0$, then $\displaystyle\max_{k\geq1}\mu_{k}(\sigma)=\mu_1(\sigma)$ and
\begin{equation} \label{eq:muk-2}
\mu_1(\sigma) - \mu_k(\sigma) \geq C_{d,\sigma}>0 \qquad\text{for all }k\geq2,
\end{equation}
for some constant $C_{d,\sigma}$ depending only on $d$ and $\sigma$.
\end{enumerate}
Furthermore, the sequence $(\mu_k(\sigma))_k$ is bounded, with
\begin{equation} \label{eq:muk-limit}
\lim_{k\to\infty}\mu_k(\sigma) =  (d-1)\omega_{d-1}2^{d-1+\sigma} \, \frac{\Gamma\left(\frac{d-1+\sigma}{2}\right)\Gamma\left(\frac{d-1}{2}\right)}{\Gamma\left(\frac{2d-2+\sigma}{2}\right)}.
\end{equation}
Finally, one has
\begin{equation} \label{eq:mu1}
\mu_1(\sigma)= c^2_{\sigma,\partial B_1}
\end{equation}
where $c^2_{\sigma,\partial B_1}$ is the constant in \eqref{eq:nlcurvature_ball}.
\end{lemma}

\begin{proof}
The behaviour of the sequence $(\mu_k)_k$ is studied in detail in \cite{Asc22} for $\sigma>0$ and in \cite{FFMMM} for $\sigma<0$. We only show for completeness the properties that are not explicitly proved in those papers.

From the definition \eqref{eq:mu} of $\mu_k(\sigma)$ and a straightforward calculation one has
\begin{equation*}
\mu_{k+1}(\sigma)-\mu_k(\sigma) = 
(d-1)\omega_{d-1}2^{d-1+\sigma} \, \frac{\Gamma\left(\frac{d-1+\sigma}{2}\right)\Gamma\left(\frac{d-1}{2}\right)}{\Gamma\left(\frac{2d-2+\sigma}{2}\right)} \cdot  \frac{(d-1+\sigma)\prod_{j=0}^{k-1}(j-\frac{\sigma}{2})}{\prod_{j=0}^{k}(j+d-1+\frac{\sigma}{2})} \,,
\end{equation*}
from which it follows immediately that if $\sigma\in(-(d-1),0)$ the sequence $\mu_k(\sigma)$ is monotone increasing; equation \eqref{eq:muk-1} is an immediate consequence. In the case $\sigma>0$, \eqref{eq:muk-2} is proved in \cite[Proposition~3.8]{Asc22}.
The limit value in \eqref{eq:muk-limit} can be easily computed from \eqref{eq:mu} by observing that
\begin{equation} \label{eq:limit_prod}
\prod_{j=0}^{k-1}\frac{j-\frac{\sigma}{2}}{j+d-1+\frac{\sigma}{2}}\longrightarrow0 \qquad\text{as }k\to\infty.
\end{equation}
Finally, formula \eqref{eq:mu1} is proved in \cite[Lemma~4.4]{Asc22} in the case $\sigma>0$ and in \cite[Proposition~7.5]{FFMMM} in the case $\sigma<0$.
\end{proof}


\subsection{The stability thresholds} \label{subsect:stab}
In view of \eqref{eq:L2norm}, \eqref{eq:seminorm}, and \eqref{eq:mu1}, we have that the quadratic form \eqref{eq:quad_form} can be expressed in terms of the Fourier coefficients of $u\in L^2(\partial B_1)$ as 
\begin{equation} \label{eq:IIvar_spherical_harm}
\mathcal{Q}\en_\gamma(u) = \sum_{k=0}^\infty \sum_{i=1}^{d(k)} \Bigl[ \bigl(\mu_1(-\alpha)-\mu_k(-\alpha)\bigr) + \gamma\bigl(\mu_1(\beta)-\mu_k(\beta) \bigr)\Bigr]\bigl(a^i_k(u)\bigr)^2.
\end{equation}
In the following theorem, which is the main result of this section, we address the issue of the stability of the unit ball $B_1$ for the family of functionals $\en_\gamma$ and $-\en_\gamma$. 

\begin{theorem}[Stability of $B_1$] \label{thm:stability}
Let $\alpha\in(0,d-1)$ and $\beta>0$, and let
\begin{align}
\gamma_* &= \gamma_*(d,\alpha,\beta) \defeq \sup_{k\geq 2}\;\frac{\mu_k(-\alpha)-\mu_1(-\alpha)}{\mu_1(\beta)-\mu_k(\beta)}\,, \label{eq:gammatilde*}\\
\gamma_{**} &= \gamma_{**}(d,\alpha,\beta) \defeq \inf_{k\geq 2}\;\frac{\mu_k(-\alpha)-\mu_1(-\alpha)}{\mu_1(\beta)-\mu_k(\beta)}\,. \label{eq:gammatilde**}
\end{align}
Then $0<\gamma_{**}<\gamma_{*}<\infty$ and the following holds:
\begin{enumerate}
\item the unit ball is a volume-constrained stable set for $\en_\gamma$ if and only if $\gamma\geq\gamma_*$;
\item the unit ball is a volume-constrained stable set for $-\en_\gamma$ if and only if $\gamma\leq\gamma_{**}$.
\end{enumerate}
Moreover, one has the explicit values
\begin{equation} \label{eq:gamma*}
\gamma_* =
\begin{cases}\displaystyle
\displaystyle \frac{\alpha(d-\alpha)(2d+2+\beta)}{\beta(d+\beta)(2d+2-\alpha)}\cdot\frac{\nla(B_1)}{\nlb(B_1)} & \text{if }\beta\geq \beta_*, \\[3ex]
\displaystyle \frac{\alpha(d-\alpha)(2d-\alpha)(d-1+\beta)}{\beta(d+\beta)(2d+\beta)(d-1-\alpha)}\cdot\frac{\nla(B_1)}{\nlb(B_1)} & \text{if }\beta<\beta_*,
\end{cases}
\end{equation}
and
\begin{equation} \label{eq:gamma**}
\gamma_{**} = \frac{\alpha(d-\alpha)}{\beta(d+\beta)}\cdot \frac{\nla(B_1)}{\nlb(B_1)} \,,
\end{equation}
where $\beta_*$ is defined in \eqref{eq:beta*}.
\end{theorem}

\begin{proof}
By \eqref{eq:quadform-2}, we have that $B_1$ is a volume-constrained stable set for $\en_\gamma$ if and only if
\begin{equation} \label{eq:stability}
\mathcal{Q}\en_\gamma(u)\geq 0 \qquad \text{for every }u\in C^\infty(\partial B_1) \text{ with }\int_{\partial B_1}u\dd\Hd=0,
\end{equation} 
and similarly $B_1$ is a volume-constrained stable set for $-\en_\gamma$ if and only if
\begin{equation} \label{eq:stability2}
\mathcal{Q}\en_\gamma(u)\leq 0 \qquad \text{for every }u\in C^\infty(\partial B_1) \text{ with }\int_{\partial B_1}u\dd\Hd=0.
\end{equation} 
Indeed, the condition \eqref{eq:stability} implies that $\delta^2\en_\gamma(B_1)[X]\geq0$ for every vector field $X$ inducing a volume-preserving flow, since for every such vector field it must be $\int_{\partial B_1} X\cdot\nu_{B_1}\dd\Hd=0$. The converse implication can be proved by arguing as in \cite[Proof of Theorem~7.1]{FFMMM}.

Hence the goal is to show that \eqref{eq:stability} is equivalent to $\gamma\geq\gamma_*$ and that \eqref{eq:stability2} is equivalent to $\gamma\leq\gamma_{**}$. This is based on the expression \eqref{eq:IIvar_spherical_harm} of $\mathcal{Q}\en_\gamma(u)$ in terms of the Fourier coefficients of $u$. Indeed, notice first that if $\int_{\partial B_1}u\dd\Hd=0$, then $a_0^1(u)=0$ and therefore the sum in \eqref{eq:IIvar_spherical_harm} starts from $k=2$ (as the term in brackets vanishes identically for $k=1$):
\begin{equation} \label{eq:IIvar_spherical_harm-b}
\mathcal{Q}\en_\gamma(u) = \sum_{k=2}^\infty \sum_{i=1}^{d(k)} \Bigl[ \bigl(\mu_1(-\alpha)-\mu_k(-\alpha)\bigr) + \gamma\bigl(\mu_1(\beta)-\mu_k(\beta) \bigr)\Bigr]\bigl(a^i_k(u)\bigr)^2
\end{equation}
for every $u\in L^2(\partial B_1)$ with $\int_{\partial B_1}u\dd\Hd=0$. In view of \eqref{eq:IIvar_spherical_harm-b}, we have that the condition \eqref{eq:stability} (which is equivalent to the stability of $B_1$ for $\en_\gamma$) is satisfied if and only if $\gamma\geq\gamma_{*}$, whereas the condition \eqref{eq:stability2} (which is equivalent to the stability of $B_1$ for $-\en_\gamma$) is satisfied if and only if $\gamma\leq\gamma_{**}$, where $\gamma_{*}$ and $\gamma_{**}$ are defined in \eqref{eq:gammatilde*} and \eqref{eq:gammatilde**} respectively.

Notice that, in view of the properties listed in Lemma~\ref{lemma:muk}, we have $\gamma_{*},\gamma_{**}\in(0,+\infty)$. To complete the proof, it only remains to prove the explicit expression of the thresholds in \eqref{eq:gamma*} and \eqref{eq:gamma**}. We postpone the technical proof of this fact to Appendix~\ref{sec:appendix}, see in particular Corollary~\ref{cor:appendix}.
\end{proof}

\begin{remark}
In Lemma~\ref{lem:appendix} we show that the supremum in \eqref{eq:gamma*} is attained for $k=3$ if $\beta \geq \beta_*$, and in the limit as $k\to\infty$ if $\beta<\beta_*$. This suggests that the loss of stability of the ball is related to a perturbation by the third mode of spherical harmonics in the first case whereas it is related to fine mixing in the latter case. It would be interesting to perform a bifurcation analysis to further investigate the symmetry breaking at the stability threshold.
\end{remark}

\begin{remark} \label{rmk:IIvar_quadraticbound}
In the case $\gamma>\gamma_{*}$, by using \eqref{eq:gammatilde*} and \eqref{eq:IIvar_spherical_harm-b} we have for every $u\in L^2(\partial B_1)$ with $\int_{\partial B_1}u\dd\Hd=0$ that
\begin{align*}
\mathcal{Q}\en_\gamma(u)
& \geq \Bigl(\frac{\gamma}{\gamma_*}-1\Bigr) \sum_{k=2}^\infty\sum_{i=1}^{d(k)} \bigl[\mu_k(-\alpha)-\mu_1(-\alpha)\bigr]\bigl(a^i_k(u)\bigr)^2 .
\end{align*}
Since $\mu_k(-\alpha)-\mu_1(-\alpha)\geq C_{d,\alpha}>0$ for all $k\geq 2$ by \eqref{eq:muk-1}, we conclude that
\begin{equation} \label{eq:IIvar_quadraticbound*}
\mathcal{Q}\en_\gamma(u) \geq \Bigl(\frac{\gamma}{\gamma_*}-1\Bigr) C_{d,\alpha}\|u\|_{L^2(\partial B_1)}^2 \qquad\qquad\text{if } \gamma>\gamma_{*}.
\end{equation}
By a similar argument we can show that 
\begin{equation} \label{eq:IIvar_quadraticbound**}
\mathcal{Q}\en_\gamma(u) \leq -\Bigl(1-\frac{\gamma}{\gamma_{**}}\Bigr) C_{d,\alpha}\|u\|_{L^2(\partial B_1)}^2 \qquad\qquad\text{if } \gamma<\gamma_{**}.
\end{equation}
\end{remark}

\begin{remark} \label{rmk:explicitGamma}
We can express the value of $\nl_\sigma(B_1)$ in terms of the $\Gamma$ function, and thus get an explicit value of the thresholds $\gamma_*$ and $\gamma_{**}$. Indeed, by \eqref{eq:nlcurvature_ball}, \eqref{eq:mu1}, and \eqref{eq:mu} we find
\begin{equation} \label{eq:energyball}
\begin{split}
\nl_\sigma(B_1)
& = \frac{d\omega_d}{(d+\sigma)(2d+\sigma)}\mu_1(\sigma) \\
& = 2^{d+\sigma} \frac{d(d-1)(d-1+\sigma)\omega_d\omega_{d-1}}{(d+\sigma)(2d+\sigma)(2d-2+\sigma)} \, \frac{\Gamma\left(\frac{d-1+\sigma}{2}\right)\Gamma\left(\frac{d-1}{2}\right)}{\Gamma\left(\frac{2d-2+\sigma}{2}\right)} \,.
\end{split}
\end{equation}
\end{remark}

\begin{remark} \label{rmk:alphalarge-2}
In the case $\alpha\in[d-1,d)$ one has $\mu_k(-\alpha)\to+\infty$ as $k\to\infty$ (see Remark~\ref{rmk:alphalarge}) and therefore $\gamma_*=+\infty$: the unit ball is \emph{never} stable for the functional $\en_\gamma$ and, given any $\gamma>0$, one can always find $u\in L^2(\partial B_1)$ such that $\mathcal{Q}\en_\gamma(u)<0$. It is worth to recall that in this case $B_1$ is not even a critical point for the relaxed problem \eqref{eq:minrelaxed}, see \cite{FraLie21}.
\end{remark}


\section{A Fuglede-type result for nearly-spherical sets}\label{sec:fuglede}

In this section we prove the local optimality, in a quantitative form, of the unit ball for the functional $\en_\gamma$ (local minimality in the case $\gamma>\gamma_{*}$ and local maximality in the case $\gamma<\gamma_{**}$) among \emph{nearly spherical sets}, i.e.\ sets $E\subset\R^d$ with $|E|=|B_1|$, barycenter at the origin, and whose boundary can be represented in the form
\begin{equation} \label{eq:nearly_spherical}
\partial E = \bigl\{ (1+u(x))x \,:\, x\in\partial B_1 \}
\end{equation}
for some $u\in L^\infty(\partial B_1)$ with $\|u\|_{L^\infty(\partial B_1)}\leq\e_0$, for some $\e_0>0$ sufficiently small. The main result of this section is the following.

\begin{theorem} \label{thm:fuglede}
Let $\alpha\in(0,d-1)$, $\beta>0$, $\gamma>0$. There exist positive constants $C_0=C_0(d,\alpha,\beta)$, $C_0'=C_0'(d,\alpha,\beta)$, and $\e_0=\e_0(d,\alpha,\beta,\gamma)$ with the following property. For every nearly-spherical set $E\subset\R^d$ as in \eqref{eq:nearly_spherical} with $|E|=|B_1|$, barycenter at the origin, and $\|u\|_{L^\infty(\partial B_1)}\leq \e_0$, it holds
\begin{align}
\en_\gamma(E) - \en_\gamma (B_1) &\geq C_0\Bigl(\frac{\gamma}{\gamma_{*}}-1\Bigr)\|u\|_{L^2(\partial B_1)}^2 \geq C_0'\Bigl(\frac{\gamma}{\gamma_{*}}-1\Bigr)|\asymm{E}{B_1}|^2 & &\text{if }\gamma > \gamma_*,\label{eq:fuglede1}\\
\en_\gamma(E) - \en_\gamma (B_1) &\leq -C_0\Bigl(1-\frac{\gamma}{\gamma_{**}}\Bigr)\|u\|_{L^2(\partial B_1)}^2 \leq -C_0'\Bigl(1-\frac{\gamma}{\gamma_{**}}\Bigr)|\asymm{E}{B_1}|^2 & & \text{if }\gamma < \gamma_{**}.\label{eq:fuglede2}
\end{align}
\end{theorem}

\begin{proof}
The proof follows the lines of \cite[Lemma~5.3]{FFMMM}. Replacing $u$ by $tu$, we can consider a generic open set $E_t\subset\R^d$ with $|E_t|=|B_1|$, barycenter at the origin, and
\begin{equation} \label{proof:fuglede-1}
\partial E_t = \bigl\{ (1+tu(x))x \,:\, x\in\partial B_1 \},
\end{equation}
where $\|u\|_{L^\infty(\partial B_1)}\leq\frac12$ and $t\in(0,2\e_0)$, with $\e_0\in(0,\frac18)$ to be chosen later.

By the same computations as in the proof of \cite[Lemma~5.3]{FFMMM} we have that, for every $\sigma\in(-d,\infty)$,
\begin{align}\label{proof:fuglede-2}
\nl_\sigma(E_t) - \nl_\sigma(B_1) = \frac{\nl_\sigma(B_1)}{d\omega_d}\Bigl(h_\sigma(t)-h_\sigma(0)\Bigr) - \frac{t^2}{2}g_\sigma(t),
\end{align}
where
\begin{align} \label{proof:fuglede-3}
h_\sigma(t)\defeq \int_{\partial B_1}\bigl(1+tu(x)\bigr)^{2d+\sigma}\dd\Hd_x,
\end{align}
\begin{align} \label{proof:fuglede-4}
g_\sigma(t)\defeq \int_{\partial B_1}\dd\Hd_x \int_{\partial B_1}\dd\Hd_y \int_{u(y)}^{u(x)}\dd r \int_{u(y)}^{u(x)} f_\sigma(1+tr,1+t\rho,|x-y|)\dd\rho,
\end{align}
and
\begin{align} \label{proof:fuglede-5}
f_\sigma(r,\rho,\theta)\defeq r^{d-1}\rho^{d-1}K_\sigma\bigl( (|r-\rho|^2+r\rho\theta^2)^\frac12 \bigr).
\end{align}
We now consider separately the two terms on the right-hand side of \eqref{proof:fuglede-2}.

\medskip\noindent\textit{Estimate on $h_\sigma$.}
Observe that the condition $|E_t|=|B_1|$ yields $h_\sigma(0)=d\omega_d=d|E_t|=\int_{\partial B_1}(1+tu)^d\dd\Hd$. Therefore,
\begin{align*}
h_\sigma&(t)-h_\sigma(0)
 =\int_{\partial B_1}\bigl(1+tu(x)\bigr)^d \bigl[\bigl(1+tu(x)\bigr)^{d+\sigma}-1\bigr] \dd\Hd_x \\
& = (d+\sigma)t\int_{\partial B_1}u\dd\Hd + \frac12(d+\sigma)(3d+\sigma-1)t^2\int_{\partial B_1} u^2\dd\Hd + \mathcal{O}_{d,\sigma}\bigl(t^3\|u^3\|_{L^1(\partial B_1)}\bigr),
\end{align*}
where $|\mathcal{O}_{d,\sigma}(s)|\leq C s$ for every $s\in[-\frac12,\frac12]$, for a constant $C$ depending only on $d$ and $\sigma$. Next, using again the condition $|E_t|=|B_1|$, we have
\begin{align*}
0 & =\int_{\partial B_1}\bigl[\bigl(1+tu(x)\bigr)^d-1\bigr]\dd\Hd_x \\
& = dt\int_{\partial B_1}u\dd\Hd + \frac{1}{2}d(d-1)t^2\int_{\partial B_1}u^2\dd\Hd + \mathcal{O}_{d}(t^3\|u^3\|_{L^1(\partial B_1)}).
\end{align*}
By combining this identity with the previous one we eventually obtain
\begin{equation} \label{proof:fuglede-6}
h_\sigma(t)-h_\sigma(0) 
= \frac{t^2}{2} (d+\sigma)(2d+\sigma)\int_{\partial B_1}u^2\dd\Hd + \mathcal{O}_{d,\sigma}\bigl(t^3\|u^3\|_{L^1(\partial B_1)}\bigr).
\end{equation}

\medskip\noindent\textit{Estimate on $g_\sigma$.}
In order to estimate $|g'_\sigma(s)|$, we compute for $r,\rho\in[-\frac12,\frac12]$, $\theta\in(0,2]$, and $s\in(0,2\e_0)$
\begin{align} \label{proof:estimate_g-0}
\bigg|&\frac{\dd}{\dd s}f_\sigma(1+sr,1+s\rho,\theta)\bigg|
	= (1+sr)^{d-1}(1+s\rho)^{d-1} \times \nonumber\\
& \quad \times \bigg| (d-1)\biggl(\frac{r}{(1+sr)}+\frac{\rho}{(1+s\rho)}\biggr) K_\sigma\bigl(\bigl(s^2|r-\rho|^2+(1+sr)(1+s\rho)\theta^2\bigr)^\frac12\bigr) \nonumber\\
&  \qquad + \frac{\sigma}{2}\frac{K_\sigma\bigl(\bigl(s^2|r-\rho|^2+(1+sr)(1+s\rho)\theta^2\bigr)^\frac12\bigr)}{s^2|r-\rho|^2+(1+sr)(1+s\rho)\theta^2} \Bigl[ 2s(r-\rho)^2+\bigl(r(1+s\rho)+\rho(1+sr)\bigr)\theta^2 \Bigr] \bigg| \nonumber\\
& \leq C_{d,\sigma}K_\sigma\bigl(\bigl(s^2|r-\rho|^2+(1+sr)(1+s\rho)\theta^2\bigr)^\frac12\bigr)
\biggl[1+\frac{2s(r-\rho)^2}{s^2|r-\rho|^2+(1+sr)(1+s\rho)\theta^2} \biggr]
\end{align}
for a uniform constant $C_{d,\sigma}$ depending only on $d$ and $\sigma$, which in the following might change from line to line.

Consider first the case $\sigma\in(-(d-1),0)$. In this case $K_\sigma$ is monotone decreasing, hence from \eqref{proof:estimate_g-0} we have
\begin{equation*}
\bigg|\frac{\dd}{\dd s}f_\sigma(1+sr,1+s\rho,\theta)\bigg| \leq C_{d,\sigma}K_\sigma(\theta)
\biggl[1+\frac{2s(r-\rho)^2}{s^2|r-\rho|^2+(1+sr)(1+s\rho)\theta^2} \biggr]\,.
\end{equation*}
Let $C_s\defeq\{(x,y)\in\partial B_1\times\partial B_1 : |x-y|\geq \sqrt{s}\}$ and $D_s\defeq\{(x,y)\in\partial B_1\times\partial B_1 : |x-y|<\sqrt{s}\}$. Then
\begin{align} \label{proof:estimate_g-1}
|&g_\sigma'(s)|
\leq C_{d,\sigma} \int_{\partial B_1}\int_{\partial B_1}K_\sigma(|x-y|)|u(x)-u(y)|^2\dd\Hd_x\dd\Hd_y \nonumber\\
& + C_{d,\sigma}\iint_{C_s}\mathrm{d}\Hd_x\mathrm{d}\Hd_y 
\int_{u(x)\land u(y)}^{u(x)\lor u(y)}\mathrm{d} r 
\int_{u(x)\land u(y)}^{u(x)\lor u(y)} \frac{2s(r-\rho)^2 K_\sigma(|x-y|) }{s^2|r-\rho|^2+(1+sr)(1+s\rho)|x-y|^2}\dd\rho \nonumber \\
& + C_{d,\sigma}\iint_{D_s}\mathrm{d}\Hd_x\mathrm{d}\Hd_y 
\int_{u(x)\land u(y)}^{u(x)\lor u(y)}\mathrm{d} r 
\int_{u(x)\land u(y)}^{u(x)\lor u(y)} \frac{2s(r-\rho)^2 K_\sigma(|x-y|) }{s^2|r-\rho|^2+(1+sr)(1+s\rho)|x-y|^2}\dd\rho \nonumber \\
& \leq C_{d,\sigma} \int_{\partial B_1}\int_{\partial B_1}K_\sigma(|x-y|)|u(x)-u(y)|^2\dd\Hd_x\dd\Hd_y \nonumber \\
& \quad + \frac{C_{d,\sigma}}{s}\iint_{D_s} K_\sigma(|x-y|)|u(x)-u(y)|^2 \dd\Hd_x\dd\Hd_y \nonumber \\
& \leq C_{d,\sigma} \int_{\partial B_1}\int_{\partial B_1}K_\sigma(|x-y|)|u(x)-u(y)|^2\dd\Hd_x\dd\Hd_y \nonumber \\
& \quad + \frac{C_{d,\sigma}}{s}\int_{\partial B_1} \biggl(\int_{\{y\in\partial B_1 \,:\, |y-x|<\sqrt{s}\}} K_\sigma(|x-y|)\dd\Hd_y\biggr) |u(x)|^2 \dd\Hd_x \nonumber \\
&= C_{d,\sigma} \int_{\partial B_1}\int_{\partial B_1}K_\sigma(|x-y|)|u(x)-u(y)|^2\dd\Hd_x\dd\Hd_y 
+ C_{d,\sigma} s^{\frac{d-1+\sigma}{2}-1} \|u\|_{L^2(\partial B_1)}^2,
\end{align}
where in the last step we used the fact that $\sigma>-(d-1)$.

Next, we consider the case $\sigma\geq 0$. In this case $K_\sigma\bigl(\bigl(s^2|r-\rho|^2+(1+sr)(1+s\rho)\theta^2\bigr)^\frac12\bigr)$ is uniformly bounded in the range of parameters that we consider. Then, arguing as in the previous case, from \eqref{proof:estimate_g-0} we deduce that
\begin{align} \label{proof:estimate_g-2}
|&g_\sigma'(s)|
\leq C_{d,\sigma}\iint_{C_s\cup D_s}
\int_{u(x)\land u(y)}^{u(x)\lor u(y)}\mathrm{d}r 
\int_{u(x)\land u(y)}^{u(x)\lor u(y)}\biggl[1+\frac{2s(r-\rho)^2}{s^2|r-\rho|^2+(1+sr)(1+s\rho)|x-y|^2} \biggr] \dd\rho \nonumber\\
& \leq C_{d,\sigma}\int_{\partial B_1}\int_{\partial B_1}|u(x)-u(y)|^2\dd\Hd_x\dd\Hd_y \nonumber \\
& \quad + C_{d,\sigma}\iint_{C_s}\dd\Hd_x\dd\Hd_y 
\int_{u(x)\land u(y)}^{u(x)\lor u(y)}\dd r 
\int_{u(x)\land u(y)}^{u(x)\lor u(y)}\frac{2|x-y|^2(r-\rho)^2}{(1+sr)(1+s\rho)|x-y|^2}\dd\rho \nonumber\\
& \quad + C_{d,\sigma}\iint_{D_s}\dd\Hd_x\dd\Hd_y 
\int_{u(x)\land u(y)}^{u(x)\lor u(y)}\dd r 
\int_{u(x)\land u(y)}^{u(x)\lor u(y)}\frac{2s(r-\rho)^2}{s^2|r-\rho|^2}\dd\rho \nonumber\\
& \leq C_{d,\sigma} \int_{\partial B_1}\int_{\partial B_1}|u(x)-u(y)|^2\dd\Hd_x\dd\Hd_y + \frac{C_{d,\sigma}}{s}\iint_{D_s} |u(x)-u(y)|^2 \dd\Hd_x\dd\Hd_y \nonumber\\
& \leq C_{d,\sigma}\|u\|_{L^2(\partial B_1)}^2 + \frac{C_{d,\sigma}}{s}\int_{\partial B_1} \biggl(\int_{\{y\in\partial B_1 \,:\, |y-x|<\sqrt{s}\}}\dd\Hd_y\biggr) |u(x)|^2 \dd\Hd_x \nonumber\\
& \leq C_{d,\sigma}\Bigl(1+s^{\frac{d-1}{2}-1}\Bigr)\|u\|_{L^2(\partial B_1)}^2 .
\end{align}

Therefore, from \eqref{proof:estimate_g-1} and \eqref{proof:estimate_g-2} we conclude that for every $\sigma>-(d-1)$ and for every $t\in(0,2\e_0)$
\begin{equation} \label{proof:fuglede-7}
\begin{split}
g_\sigma(t)
& = g_\sigma(0) + \int_0^t g'_\sigma(s)\dd s \\
& = \int_{\partial B_1}\int_{\partial B_1}K_\sigma(|x-y|)|u(x)-u(y)|^2\dd\Hd_x\dd\Hd_y + \mathcal{R}_\sigma(t),
\end{split}
\end{equation}
with
\begin{equation} \label{proof:fuglede-7b}
\begin{split}
|\mathcal{R}_\sigma(t)|
& \leq C_{d,\sigma}t\int_{\partial B_1}\int_{\partial B_1}K_\sigma(|x-y|)|u(x)-u(y)|^2\dd\Hd_x\dd\Hd_y \\
& \qquad + C_{d,\sigma}t\|u\|_{L^2(\partial B_1)}^2 + C_{d,\sigma}\tilde\omega(t)\|u\|_{L^2(\partial B_1)}^2,
\end{split}
\end{equation}
and
$$
\tilde\omega(t)\coloneqq
\begin{cases}
t^{\frac{d-1+\sigma}{2}} & \text{ if } \sigma< 0,\\
t^{\frac{d-1}{2}} & \text{ if } \sigma\geq 0.
\end{cases}
$$

\medskip\noindent\textit{Conclusion of the proof.}
By inserting \eqref{proof:fuglede-6} and \eqref{proof:fuglede-7} into \eqref{proof:fuglede-2} we obtain
\begin{equation} \label{proof:fuglede-8}
\begin{split}
\nl_\sigma&(E_t) -\nl_\sigma(B_1) = \frac{t^2}{2}\biggl[ \frac{(d+\sigma)(2d+\sigma)}{d\omega_d}\,\nl_\sigma(B_1)\int_{\partial B_1}u^2\dd\Hd \\
&- \int_{\partial B_1}\int_{\partial B_1}K_\sigma(|x-y|)|u(x)-u(y)|^2\dd\Hd_x\dd\Hd_y \biggr]
+ \mathcal{O}_{d,\sigma}\bigl(t^3\|u^3\|_{L^1(\partial B_1)}\bigr) + t^2 \mathcal{R}_\sigma(t),
\end{split}
\end{equation}
where $\mathcal{R}_\sigma(t)$ obeys the estimate \eqref{proof:fuglede-7b}.

Recalling \eqref{eq:quad_form} and \eqref{eq:nlcurvature_ball}, in the term in brackets we recognize the quadratic form associated to the second variation of the functional $\nl_\sigma$. Hence by combining the formula \eqref{proof:fuglede-8} in the two cases $\sigma=-\alpha$ and $\sigma=\beta$ we obtain for the full energy $\en_\gamma$
\begin{align} \label{proof:fuglede-9}
\en_\gamma(E_t)-\en_\gamma(B_1) = \frac{t^2}{2}\mathcal{Q}\en_\gamma(u) + \mathcal{R}(t),
\end{align}
with
\begin{align} \label{proof:fuglede-10}
|\mathcal{R}(t)| & \leq C_{d,\alpha,\beta}t^2\omega(t)\biggl( \int_{\partial B_1}\int_{\partial B_1}\frac{|u(x)-u(y)|^2}{|x-y|^\alpha}\dd\Hd_x\dd\Hd_y \nonumber \\
& \qquad + \gamma\int_{\partial B_1}\int_{\partial B_1}|x-y|^\beta|u(x)-u(y)|^2\dd\Hd_x\dd\Hd_y + (1+\gamma)\|u\|^2_{L^2(\partial B_1)}\biggr)
\end{align}
for a uniform constant $C_{d,\alpha,\beta}$ depending only on $d$, $\alpha$, $\beta$, which in the following might change from line to line, and some modulus of continuity $\omega(t)\to0$ as $t\to0$, also uniform in $t$ and $u$.

To conclude the proof, we rely once again on the spherical harmonic representation of the previous quantities. Denoting by $a^i_k(u)$ the Fourier coefficients of $u$ as in \eqref{eq:fourier}, we preliminary observe that exploiting the condition $|E_t|=|B_1|$ we have
\begin{equation} \label{proof:fuglede-11}
|a_0^1(u)| = \bigg| \frac{1}{\sqrt{d\omega_d}}\int_{\partial B_1}u\dd\Hd \bigg| \leq C_d t \|u\|_{L^2(\partial B_1)}^2,
\end{equation}
and similarly since the barycenter of $E_t$ is at the origin we have $\int_{\partial B_1}(1+tu)^{d+1}x\dd\Hd_x=0$, which yields
\begin{equation} \label{proof:fuglede-12}
|a_1^i(u)| = \bigg| \frac{1}{\sqrt{\omega_d}}\int_{\partial B_1}u\,x_i\dd\Hd \bigg| \leq C_d t \|u\|_{L^2(\partial B_1)}^2 \qquad\text{for }i=1,\ldots,d.
\end{equation}

We first consider the case $\gamma>\gamma_*$. Using \eqref{eq:L2norm}, \eqref{eq:seminorm} and \eqref{eq:IIvar_spherical_harm}, and inserting them into \eqref{proof:fuglede-9}--\eqref{proof:fuglede-10}, we obtain
\begin{align*}
\en_\gamma(E_t) & - \en_\gamma(B_1) \\
& \geq \frac{t^2}{2}\sum_{k=0}^\infty\sum_{i=1}^{d(k)} \Bigl[ \bigl(\mu_1(-\alpha)-\mu_k(-\alpha)\bigr) + \gamma\bigl(\mu_1(\beta)-\mu_k(\beta) \bigr)\Bigr]\bigl(a^i_k(u)\bigr)^2 \\
& \qquad - C_{d,\alpha,\beta} t^2\omega(t) \sum_{k=0}^\infty\sum_{i=1}^{d(k)} \Bigl[\mu_k(-\alpha)+\gamma\mu_k(\beta)+(1+\gamma)\Bigr] \bigl(a^i_k(u)\bigr)^2 \\
& \geq \frac{t^2}{2}\bigl(1+C_{d,\alpha,\beta}\omega(t)\bigr)\sum_{k=0}^\infty\sum_{i=1}^{d(k)} \Bigl[ \bigl(\mu_1(-\alpha)-\mu_k(-\alpha)\bigr) + \gamma\bigl(\mu_1(\beta)-\mu_k(\beta) \bigr)\Bigr]\bigl(a^i_k(u)\bigr)^2 \\
& \qquad - C_{d,\alpha,\beta} t^2\omega(t) \Bigl[\mu_1(-\alpha)+\gamma\mu_1(\beta)+(1+\gamma)\Bigr] \|u\|_{L^2(\partial B_1)}^2 .
\end{align*}
Recalling that $\mu_0(-\alpha)=\mu_0(\beta)=0$ and in view of the expression \eqref{eq:gammatilde*} of $\gamma_*$, we then obtain
\begin{align*}
\en_\gamma(E_t) - \en_\gamma(B_1) 
& \geq \frac{t^2}{2}\bigl(1+C_{d,\alpha,\beta}\omega(t)\bigr)\Bigl(\frac{\gamma}{\gamma_*}-1\Bigr) \sum_{k=2}^\infty\sum_{i=1}^{d(k)} \bigl[\mu_k(-\alpha)-\mu_1(-\alpha)\bigr]\bigl(a^i_k(u)\bigr)^2 \\
& \qquad - C_{d,\alpha,\beta} t^2\omega(t) \Bigl[ \mu_1(-\alpha)+\gamma\mu_1(\beta)+ 1 +\gamma \Bigr]\|u\|_{L^2(\partial B_1)}^2 .
\end{align*}
By Lemma~\ref{lemma:muk}, $\mu_k(-\alpha)-\mu_1(-\alpha)\geq C_{d,\alpha}\mu_1(-\alpha)$ for all $k\geq 2$ with $C_{d,\alpha}>0$, hence
\begin{align*}
\en_\gamma(E_t) & - \en_\gamma(B_1) \\
& \geq \frac{t^2}{2}\bigl(1+C_{d,\alpha,\beta}\omega(t)\bigr)\Bigl(\frac{\gamma}{\gamma_*}-1\Bigr)C_{d,\alpha}\mu_1(-\alpha) \sum_{k=2}^\infty \bigl(a^i_k(u)\bigr)^2 \nonumber \\
& \qquad - C_{d,\alpha,\beta} t^2\omega(t) \Bigl[ \mu_1(-\alpha)+\gamma\mu_1(\beta)+ 1 +\gamma \Bigr]\|u\|_{L^2(\partial B_1)}^2 \\
& = \frac{t^2}{2}\bigl(1+C_{d,\alpha,\beta}\omega(t)\bigr)\Bigl(\frac{\gamma}{\gamma_*}-1\Bigr)C_{d,\alpha}\mu_1(-\alpha) \biggl(\|u\|_{L^2(\partial B_1)}^2 - |a_0^1(u)|^2-\sum_{i=1}^d|a_1^i(u)|^2\biggr)\\
& \qquad - C_{d,\alpha,\beta} t^2\omega(t) \Bigl[ \mu_1(-\alpha)+\gamma\mu_1(\beta)+ 1 +\gamma \Bigr]\|u\|_{L^2(\partial B_1)}^2 .
\end{align*}
Thanks to \eqref{proof:fuglede-11} and \eqref{proof:fuglede-12}, and recalling that $\omega(t)\to0$ as $t\to0$, it is now easy to conclude that \eqref{eq:fuglede1} holds in the case $\gamma>\gamma_*$ provided that we choose $\e_0$ small enough.

The argument in the case $\gamma<\gamma_{**}$ is very similar and we will only sketch here the main differences with the previous one. Starting once again from \eqref{proof:fuglede-9}--\eqref{proof:fuglede-10} and using the spherical harmonics representation we have
\begin{align*}
\en_\gamma(E_t) - \en_\gamma(B_1) 
& \leq \frac{t^2}{2}\sum_{k=0}^\infty\sum_{i=1}^{d(k)} \Bigl[ \bigl(\mu_1(-\alpha)-\mu_k(-\alpha)\bigr) + \gamma\bigl(\mu_1(\beta)-\mu_k(\beta) \bigr)\Bigr]\bigl(a^i_k(u)\bigr)^2 \\
& \qquad + C_{d,\alpha,\beta} t^2\omega(t) \sum_{k=0}^\infty\sum_{i=1}^{d(k)} \Bigl[\mu_k(-\alpha)+\gamma\mu_k(\beta)+(1+\gamma)\Bigr] \bigl(a^i_k(u)\bigr)^2. \end{align*}
By similar computations as before and using the representation \eqref{eq:gammatilde**} of $\gamma_{**}$ we then find
\begin{align*}
\en_\gamma(E_t) - \en_\gamma(B_1) 
& \leq \frac{t^2}{2}\bigl(1-C_{d,\alpha,\beta}\omega(t)\bigr)\Bigl(1-\frac{\gamma}{\gamma_{**}}\Bigr) \sum_{k=2}^\infty\sum_{i=1}^{d(k)} \bigl[\mu_1(-\alpha)-\mu_k(-\alpha)\bigr]\bigl(a^i_k(u)\bigr)^2 \\
& \qquad + C_{d,\alpha,\beta} t^2\omega(t) \Bigl[ \mu_1(-\alpha)+\gamma\mu_1(\beta)+ 1 +\gamma \Bigr]\|u\|_{L^2(\partial B_1)}^2,
\end{align*}
and eventually, using $\mu_k(-\alpha)-\mu_1(-\alpha)\geq C_{d,\alpha}\mu_1(-\alpha)$, \eqref{proof:fuglede-11}, and \eqref{proof:fuglede-12}, we conclude as before that \eqref{eq:fuglede2} holds.
\end{proof}


\section{Local optimality of the ball}\label{sec:locmin}

Having showed the local optimality of the ball among nearly-spherical sets in Theorem~\ref{thm:fuglede}, the next step consists in proving that $B_1$ is a local minimizer (resp. maximizer) of $\en_\gamma$, for $\gamma>\gamma_{*}$ (resp. $\gamma<\gamma_{**}$), among all measurable sets $E\subset \R^d$ with $|E|=|B_1|$ and whose boundary is uniformly close to the boundary of the ball, in the sense that they satisfy the inclusions
\begin{equation} \label{eq:linfty-neighbour}
B_{1-\e_1}(x_0) \subset E \subset B_{1+\e_1}(x_0)
\end{equation}
for some $x_0\in\R^d$ and $\e_1>0$ sufficiently small. More precisely, we will prove the following result.

\begin{theorem} \label{thm:locmin_infinity}
Let $\alpha\in(0,d-1)$, $\beta>0$, $\gamma>0$. Then, there exist $C_1>0$ and $\e_1>0$ (depending only on $d$, $\alpha$, $\beta$, $\gamma$) such that
\begin{align}
\en_\gamma(E) - \en_\gamma(B_1) \geq C_1 (\Asymm{E})^2 & \qquad\text{if }\gamma>\gamma_{*}, \label{eq:locmin_infinity1}\\
\en_\gamma(E) - \en_\gamma(B_1) \leq - C_1 (\Asymm{E})^2 & \qquad\text{if }\gamma<\gamma_{**}, \label{eq:locmin_infinity2}
\end{align}
for every measurable set $E\subset\R^d$ with $|E|=|B_1|$ and satisfying \eqref{eq:linfty-neighbour} for some $x_0\in\R^d$, where $\Asymm{E}$ denotes the asymmetry defined in \eqref{eq:asymm}.
\end{theorem}

For the proof, we need to pass from the optimality of the ball among nearly-spherical sets, proved in Theorem~\ref{thm:fuglede}, to the larger class of sets uniformly close to the ball. The proof follows the construction in \cite[Proposition~2.10]{FusPra20}, where the same strategy is used to prove the quantitative maximality of the ball for the Riesz energy $E\mapsto \int_E\int_E \frac{1}{|x-y|^\alpha}$ (see also \cite[Proposition~5.12]{Asc22}, where the same is done for the attractive term $\int_E\int_E|x-y|^\beta$). The proof is based on a transport argument and on estimates on the variation of the energy which are strongly dependent on the monotonicity of the potential of the ball (defined in \eqref{eq:potential_ball} below). The combination of the two energies destroys the global monotonicity of the potential; however, what is actually needed is just the strict monotonicity in a neighbourhood of $\partial B_1$, which is still true in the case $\gamma\in(0,\gamma_{**})\cup(\gamma_{*},\infty)$, as Lemma~\ref{lem:potential} below shows. In view of this property, the proof of Theorem~\ref{thm:locmin_infinity} is almost identical to that in \cite{FusPra20,Asc22}.

\begin{lemma} \label{lem:potential}
Let $\alpha\in(0,d-1)$, $\beta>0$, $\gamma>0$, and define the \emph{potential of the unit ball} as the function
\begin{equation} \label{eq:potential_ball}
\psi(x) \coloneqq \int_{B_1}\frac{\dd y}{|x-y|^\alpha} + \gamma\int_{B_1}|x-y|^\beta \dd y.
\end{equation}
Then $\psi$ is a radial function of class $C^1$ and, writing with abuse of notation $\psi(x)=\psi(|x|)$,
\begin{align}
\psi'(1)>0 & \qquad\text{if }\gamma>\gamma_{*}, \label{eq:potential-derivative1}\\
\psi'(1)<0 & \qquad\text{if }\gamma<\gamma_{**}. \label{eq:potential-derivative2}
\end{align}
\end{lemma}

\begin{proof}
The $C^1$-regularity of the first term in \eqref{eq:potential_ball} is standard and is proved, for instance, in \cite[Proposition~2.1]{BC14} (notice that here the assumption $\alpha<d-1$ is essential); the regularity of the second term is straightforward.

We now express the quadratic form $\mathcal{Q}\en_\gamma$ associated to the second variation of the functional, see \eqref{eq:quad_form}, in terms of the derivative of $\psi$. Recalling \eqref{eq:nlcurvature} we have for all $u\in L^2(\partial B_1)$
\begin{align*}
\mathcal{Q}\en_\gamma(u) & = - \int_{\partial B_1}\int_{\partial B_1} \biggl(\frac{1}{|x-y|^\alpha}+\gamma|x-y|^\beta\biggr)|u(x)-u(y)|^2 \dd\Hd_x\dd\Hd_y \\
& \qquad + \int_{\partial B_1} \Biggl( \int_{\partial B_1} \biggl(\frac{1}{|x-y|^\alpha}+\gamma|x-y|^\beta\biggr) |x-y|^2\dd\Hd_y \Biggr)(u(x))^2\dd\Hd_x\\
& = 2\int_{\partial B_1}\int_{\partial B_1} \biggl(\frac{1}{|x-y|^\alpha}+\gamma|x-y|^\beta\biggr)u(x)u(y) \dd\Hd_x\dd\Hd_y \\
& \qquad -2 \int_{\partial B_1} \Biggl( \int_{\partial B_1} \biggl(\frac{1}{|x-y|^\alpha}+\gamma|x-y|^\beta\biggr) x\cdot y\dd\Hd_y \Biggr)(u(x))^2\dd\Hd_x.
\end{align*}
Observing now that for $x\in\partial B_1$, by the divergence theorem
\begin{align*}
\int_{\partial B_1} \biggl(\frac{1}{|x-y|^\alpha}+\gamma|x-y|^\beta\biggr) x\cdot y\dd\Hd_y
& = \int_{B_1} \nabla_y \biggl(\frac{1}{|x-y|^\alpha}+\gamma|x-y|^\beta\biggr)\cdot x \dd y \\
& = -\int_{B_1} \nabla_x \biggl(\frac{1}{|x-y|^\alpha}+\gamma|x-y|^\beta\biggr)\cdot x \dd y \\
& = -\nabla\psi(x)\cdot x = -\psi'(1),
\end{align*}
we obtain
\begin{equation} \label{eq:IIvar-potential}
\mathcal{Q}\en_\gamma(u) = 2\psi'(1)\|u\|_{L^2(\partial B_1)}^2 + 2\int_{\partial B_1}\int_{\partial B_1} \biggl(\frac{1}{|x-y|^\alpha}+\gamma|x-y|^\beta\biggr)u(x)u(y) \dd\Hd_x\dd\Hd_y.
\end{equation}

By standard arguments, it is possible to construct a sequence $(u_n)_{n\in\N}\subset L^\infty(\partial B_1)$ such that $\int_{\partial B_1}u_n\dd\Hd=0$, $\|u_n\|_{L^2(\partial B_1)}=1$, and
$$
\lim_{n\to\infty}\int_{\partial B_1}\int_{\partial B_1} \biggl(\frac{1}{|x-y|^\alpha}+\gamma|x-y|^\beta\biggr)u_n(x)u_n(y) \dd\Hd_x\dd\Hd_y = 0.
$$ 
In view of this property, the conclusion of the lemma follows by \eqref{eq:IIvar-potential} and by the quadratic bounds \eqref{eq:IIvar_quadraticbound*}--\eqref{eq:IIvar_quadraticbound**}.
\end{proof}

\medskip

\begin{remark} \label{rmk:ball-noncrit}
The unit ball $B_1$ is a critical point for the relaxed problem \eqref{eq:minrelaxed} provided 
	\begin{equation} \label{eq:ball-crit}
	 \psi(r_1)\leq \psi(1) \leq \psi(r_2) \quad \text{for} \quad r_1<1<r_2.
	\end{equation}
When $\alpha\in(0,d-1)$, for large $\gamma$, the attractive part dominates and $\psi^\pr(1)>0$ as shown above, suggesting that \eqref{eq:ball-crit} holds. This is indeed proved in \cite[Lemma~3.1]{FraLie21}. In the case $\alpha\in [d-1,d)$, instead, the derivative of the potential of the repulsive term blows up near the boundary of $B_1$. Due to this fact, no matter how large $\gamma$ is the potential $\psi$ has the opposite monotonicity than required by \eqref{eq:ball-crit}, and therefore the ball is never a critical point of \eqref{eq:minrelaxed} (see \cite[Remark~3.2]{FraLie21} for details).
\end{remark}

Thanks to the local strict monotonicity of $\psi$ in a neighbourhood of $\partial B_1$ proved in Lemma~\ref{lem:potential}, the proof of Theorem~\ref{thm:locmin_infinity} follows by repeating line by line the argument in \cite{FusPra20}. We just sketch here the general strategy discussing only the modifications needed in our setting.

\begin{lemma} \label{lemma:locmin_infinity}
There exist $\tilde{c}>0$ and $\tilde{\e}>0$ (depending on $d$, $\alpha\in(0,d-1)$, $\beta>0$, $\gamma>0$) with the following property. For every $\e\in(0,\tilde{\e})$ and for every measurable set $E\subset\R^d$ with $|E|=|B_1|$ and $B_{1-\e}(z)\subset E \subset B_{1+\e}(z)$, $z\in\R^d$, there exist functions $u^{\pm}:\partial B_1\to[0,\e)$ such that the set
$$
E_1 \defeq \Bigl\{ z + tx \,:\, x\in\partial B_1, \, t\in[0,1-u^-(x))\cup(1,1+u^+(x))\Bigr\}
$$
satisfies $|E_1|=|B_1|$ and
\begin{align}
\en_\gamma(E_1) \leq \en_\gamma(E) & \qquad\text{if }\gamma>\gamma_{*}, \label{eq:locmin_infinity-1}\\
\en_\gamma(E_1) \geq \en_\gamma(E) & \qquad\text{if }\gamma<\gamma_{**}. \label{eq:locmin_infinity-2}
\end{align}
Furthermore, either
\begin{equation} \label{eq:locmin_infinity-3}
\Asymm{E_1}\geq\frac12\Asymm{E},
\end{equation}
or
\begin{align}
\en_\gamma(E) - \en_\gamma(E_1) \geq \tilde{c} \,\Asymm{E}^2 & \qquad\text{if }\gamma>\gamma_{*}, \label{eq:locmin_infinity-4}\\
\en_\gamma(E_1) - \en_\gamma(E) \geq \tilde{c} \,\Asymm{E}^2 & \qquad\text{if }\gamma<\gamma_{**}. \label{eq:locmin_infinity-5}
\end{align}
\end{lemma}

\begin{proof}[Proof (sketch)]
The construction of the set $E_1$ is the same as in \cite[Lemma~2.11]{FusPra20}, \cite[Lemma~5.11]{Asc22}, and it is straightforward to check that the same proof can be reproduced in our setting with minor changes.

In particular, by using Lemma~\ref{lem:potential} we deduce that there exists $\tilde{\e}>0$ such that $\inf\{\psi'(r)\,:\, r\in(1-\tilde{\e},1+\tilde{\e})\}>0$ if $\gamma>\gamma_{*}$, and $\sup\{\psi'(r)\,:\, r\in(1-\tilde{\e},1+\tilde{\e})\}<0$ if $\gamma<\gamma_{**}$. This is what is needed to obtain the estimate \cite[(2.30)]{FusPra20} (in the case $\gamma<\gamma_{**}$, whereas in the case $\gamma>\gamma_{*}$ one has the opposite inequality).

Moreover, the bounds \cite[(2.29)]{FusPra20} are valid also in our setting, using \cite[Lemma~2.4]{FusPra20} and \cite[Lemma~5.5]{Asc22}.
\end{proof}

\begin{proof}[Proof of Theorem~\ref{thm:locmin_infinity} (sketch)]
Fix $\e<\tilde{\e}$, where $\tilde{\e}$ is given by Lemma~\ref{lemma:locmin_infinity} (later in the proof we will posssibly reduce the value of $\e$). Let $E\subset\R^d$ with $|E|=|B_1|$ satisfy
\begin{equation*}
B_{1-\e^2}(x_0) \subset E \subset B_{1+\e^2}(x_0)
\end{equation*}
(without loss of generality we assume $x_0=0$).

Fix now any $z\in\R^d$ with $|z|<\frac{\e}{2}$. We have that $B_{1-\e}(z)\subset E\subset B_{1+\e}(z)$ with $\e<\tilde{\e}$, and we can therefore construct the set $E_1=E_1(z)$ using Lemma~\ref{lemma:locmin_infinity}. We want to further modify $E_1$ in order to obtain a nearly-spherical set, so that we can exploit the result in Theorem~\ref{thm:fuglede}. We also fix a small number $\eta>0$, and we divide the construction into three steps.

\medskip\noindent\textit{Step 1: replace $u^{\pm}$ by locally constant functions.} By approximation, we can find two new functions $\tilde{u}^\pm:\partial B_1\to[0,\tilde{\e})$ as close as we want to $u^\pm$ in $L^\infty(\partial B_1)$, such that $\partial B_1$ can be decomposed as a disjoint union of finitely many measurable sets $\partial B_1=\bigcup_i U_i$, $\tilde{u}^\pm\equiv u^\pm_i$ is constant on each $U_i$, and $\diam(U_i)\leq \min\{u^+_i,u^-_i\}$ for each $i$ such that $\min\{u^+_i,u^-_i\}>0$. Moreover, by defining the set $E_2=E_2(z)$ as 
$$
E_2 \defeq \Bigl\{ z + tx \,:\, x\in\partial B_1, \, t\in[0,1-\tilde{u}^-(x))\cup(1,1+\tilde{u}^+(x))\Bigr\},
$$
we can guarantee that $|E_2|=|B_1|$, 
\begin{equation} \label{eq:locmin_infinity-6}
\Asymm{E_2}\geq\frac13\Asymm{E_1}, \qquad \en_\gamma(E_1)-\eta \leq \en_\gamma(E_2) \leq \en_\gamma(E_1)+\eta.
\end{equation}

\medskip\noindent\textit{Step 2: construction of the nearly-spherical set.}
We next define a function $u:\partial B_1\to(-\e,\e)$ on each set $U_i$ as follows. If $\min\{u^+_i,u^-_i\}=0$, we set $u\defeq u^+_i$ if $u^-_i=0$, and $u\defeq -u^-_i$ if $u^+_i=0$. If instead $\min\{u^+_i,u^-_i\}>0$, we write $U_i$ as the disjoint union of two sets $L_i$, $R_i$ such that $\Hd(L_i)(1-(1-u_i^-)^d)=\Hd(R_i)((1+u_i^+)^d-1)$ and we define $u\defeq \chi_{L_i}u_i^+ - \chi_{R_i}u_i^-$ on $U_i=L_i\cup R_i$.

We then let $E_3=E_3(z)$ be the nearly-spherical set around $B_1(z)$ corresponding to $u$, according to \eqref{eq:nearly_spherical}. Then by construction $|E_3|=|B_1|$ and, arguing as in the proof of \cite[Proposition~2.10]{FusPra20} (see also \cite[Proposition~5.12]{Asc22})
\begin{equation} \label{eq:locmin_infinity-7}
|\asymm{E_3}{B_1(z)}| \geq \frac12|\asymm{E_2}{B_1(z)}| \geq \frac12\Asymm{E_2} \xupref{eq:locmin_infinity-6}{\geq} \frac{1}{6}\Asymm{E_1},
\end{equation}
and (by possibly reducing the value of $\e$)
\begin{align}
\en_\gamma(E_3) - \en_\gamma(E_2) \leq 0 & \qquad\text{if }\gamma>\gamma_{*}, \label{eq:locmin_infinity-8}\\
\en_\gamma(E_3) - \en_\gamma(E_2) \geq 0 & \qquad\text{if }\gamma<\gamma_{**}. \label{eq:locmin_infinity-9}
\end{align}

\medskip\noindent\textit{Step 3: adjustment of the barycenter.}
The sets $E_1(z)$, $E_2(z)$, $E_3(z)$ constructed in the previous steps depend continuously on $z\in B_{\e/2}$. In order to apply Theorem~\ref{thm:fuglede}, we need a final modification of the set $E_3$ so that its barycenter coincides with $z$. This can be done as in \cite[Lemma~2.13 and Lemma~2.14]{FusPra20}: by the same argument, it is possible to find $z$ with $|z|<\frac{\e}{2}$ such that the barycenter of $E_3(z)$ is $z$ itself.

\medskip\noindent\textit{Conclusion of the proof.}
Summing up what we have done so far, choosing $\e_1=\e^2$ and starting from a set $E$ satisfying \eqref{eq:linfty-neighbour} (without loss of generality $x_0=0$) we have consecutively constructed three sets $E_1$, $E_2$, $E_3$, with $E_3$ nearly-spherical around a ball $B_1(z)$ and with barycenter at the point $z$. In particular, Theorem~\ref{thm:fuglede} applies to $E_3-z$.

Let us conclude the proof by showing \eqref{eq:locmin_infinity1} in the case $\gamma>\gamma_{*}$. We distinguish two cases: in the first case, the set $E_1$ (constructed by Lemma~\ref{lemma:locmin_infinity}) satisfies \eqref{eq:locmin_infinity-3}. Then
\begin{align*}
\en_\gamma(E) - \en_\gamma(B_1)
& \xupref{eq:locmin_infinity-1}{\geq} \en_\gamma(E_1)-\en_\gamma(B_1)
\xupref{eq:locmin_infinity-6}{\geq} \en_\gamma(E_2)-\en_\gamma(B_1) -\eta \\
& \xupref{eq:locmin_infinity-8}{\geq} \en_\gamma(E_3)-\en_\gamma(B_1) -\eta
\xupref{eq:fuglede1}{\geq} C_0'\Bigl(\frac{\gamma}{\gamma_{*}}-1\Bigr) |\asymm{E_3}{B_1(z)}|^2 - \eta \\
& \xupref{eq:locmin_infinity-7}{\geq} \frac{C_0'}{36}\Bigl(\frac{\gamma}{\gamma_{*}}-1\Bigr) \Asymm{E_1}^2 - \eta
\xupref{eq:locmin_infinity-3}{\geq} \frac{C_0'}{144}\Bigl(\frac{\gamma}{\gamma_{*}}-1\Bigr) \Asymm{E}^2 - \eta .
\end{align*}
Since $\eta>0$ is arbitrary, we obtain \eqref{eq:locmin_infinity1}. If instead the set $E_1$ does not satisfy \eqref{eq:locmin_infinity-3}, then \eqref{eq:locmin_infinity-4} holds. In this case
\begin{align*}
\en_\gamma(E) - \en_\gamma(B_1)
& \xupref{eq:locmin_infinity-4}{\geq} \en_\gamma(E_1)-\en_\gamma(B_1) + \tilde{c} \, \Asymm{E}^2 \\
& \xupref{eq:locmin_infinity-6}{\geq} \en_\gamma(E_2)-\en_\gamma(B_1) -\eta + \tilde{c} \, \Asymm{E}^2 \\
& \xupref{eq:locmin_infinity-8}{\geq} \en_\gamma(E_3)-\en_\gamma(B_1) -\eta + \tilde{c} \, \Asymm{E}^2 \\
& \xupref{eq:fuglede1}{\geq} - \eta + \tilde{c} \, \Asymm{E}^2,
\end{align*}
and again we obtain \eqref{eq:locmin_infinity1} since $\eta>0$ is arbitrary.
This completes the proof of \eqref{eq:locmin_infinity1}; the proof of \eqref{eq:locmin_infinity2} is completely analogous, using the corresponding estimates valid in the case $\gamma<\gamma_{**}$.
\end{proof}

\begin{proof}[Proof of Theorem~\ref{thm:main}]
By scaling, if $|E|=m$ then the set $\widetilde{E}\defeq\bigl(\frac{\omega_d}{m}\bigr)^{1/d}E$ satisfies $|\widetilde{E}|=|B_1|$ and
\begin{equation*}
\en_1(E) = \Bigl(\frac{m}{\omega_d}\Bigr)^{2-\frac{\alpha}{d}}\en_\gamma(\widetilde{E}) \qquad\text{with}\quad \gamma\defeq \Bigl(\frac{m}{\omega_d}\Bigr)^{\frac{\alpha+\beta}{d}}.
\end{equation*}
In view of this identity the theorem follows by Theorem~\ref{thm:stability} and Theorem~\ref{thm:locmin_infinity}.
\end{proof}


\section{A necessary condition for \texorpdfstring{$L^1$}{L1}-local minimality}\label{sec:L1minimality}

A natural question is whether the stability of the ball for $\gamma>\gamma_{*}$ implies its local minimality among small $L^1$-perturbations, that is whether $\en_\gamma(E)>\en_\gamma(B_1)$ for all measurable set $E\subset\R^d$ such that $|E|=|B_1|$ and $\Asymm{E}$ is sufficiently small.

An argument by Frank and Lieb \cite{FraLie21} allows to pass from sets with small asymmetry to sets which are uniformly close to the ball (in the sense of \eqref{eq:unif-ball}), and this idea is crucial in their proof of the global minimality of the ball for large $\gamma$. A similar argument is also used by Fusco and Pratelli \cite{FusPra20} for their proof of the quantitative Riesz inequality. However, both arguments are based on the \emph{global} monotonicity of the potential $\psi$ of the unit ball (see \eqref{eq:potential_ball}), which is valid only in the asymptotic regime $\gamma\gg1$ and is not necessarily true for all $\gamma>\gamma_*$.

Notice that, for $\gamma<\gamma_{**}$, it is never possible to extend the local maximality of the ball to an $L^1$-neighbourhood, as observed in the Introduction. 

The following proposition gives a necessary condition for the unit ball to be an $L^1$-local minimizer in terms of its potential $\psi$. In Remark~\ref{rmk:L1_min} we show that there are values of the parameters $\alpha$ and $\beta$ such that this necessary condition is \emph{not} satisfied for some $\gamma>\gamma_{*}$: hence, in these cases the unit ball is a stable set, but is not a local minimizer with respect to the $L^1$-topology.

\begin{proposition}\label{prop:L1-min}
Let $\psi$ be the potential of the unit ball, defined in \eqref{eq:potential_ball}. Assume that there exists $r\in(0,1)$ such that $\psi(r)>\psi(1)$, or that there exists $R>1$ such that $\psi(R)<\psi(1)$. Then for every $\delta>0$ there exists a set $E\subset\R^d$ such that $|E|=|B_1|$, $|\asymm{E}{B_1}|<\delta$, and $\en_\gamma(E)<\en_\gamma(B_1)$. In particular, $B_1$ is not an $L^1$-local minimizer.
\end{proposition}

\begin{proof}
Assume that $\psi(r)>\psi(1)$ for some $r\in(0,1)$. Then we can find $\eta>0$ and $\e_1,\e_2>0$ (with $\e_1+\e_2<1-r$) such that
\begin{equation} \label{proof:L1-min1}
\min_{|\rho-r|\leq\e_1}\psi(\rho) \geq \max_{|\rho-1|\leq \e_2}\psi(\rho) + \eta.
\end{equation}

Obviously, we only need to prove the statement for all $\delta$ sufficiently small. Let $\delta>0$ be such that $\delta<\min\{\e_1,\frac{\e_2}{2}\}$ and consider the two balls $D_1\coloneqq B_{\delta}(re_1)$, $D_2\coloneqq B_\delta((1+\delta)e_1)$. Notice that $D_1\subset B_1$ and $D_2\cap B_1=\emptyset$, and that $r-\e_1<|x|<r+\e_1$ for every $x\in D_1$, and $1<|y|<1+\e_2$ for every $y\in D_2$.

We claim that the set $E\coloneqq (B_1\setminus D_1)\cup D_2$ satisfies the conditions in the statement. Indeed, clearly $|E|=|B_1|$ and $|\asymm{E}{B_1}|=2|B_\delta|$. Moreover, denoting by
\begin{equation*}
\I(F,G) \coloneqq \int_F\int_G \biggl( \frac{1}{|x-y|^\alpha} + \gamma |x-y|^\beta \biggr)\dd x\dd y,
\qquad
\I(F) \coloneqq \I(F,F),
\end{equation*}
we have that
\begin{align*}
\en_\gamma(E)-\en_\gamma(B_1)
& = \I(B_1\setminus D_1) + \I(D_2) + 2\I(B_1\setminus D_1,D_2) - \I(B_1) \\
& = \I(B_1) - \I(D_1) - 2\I(B_1\setminus D_1,D_1) + \I(D_2) + 2\I(B_1\setminus D_1,D_2) - \I(B_1) \\
& = 2\I(B_1\setminus D_1,D_2) - 2\I(B_1\setminus D_1,D_1) \\
& = 2\I(B_1,D_2)-2\I(B_1,D_1) - 2\I(D_1,D_2) + 2\I(D_1).
\end{align*}
Observe now that $\I(D_1,D_2)\geq0$ and that $\I(D_1)\leq C(d,\alpha,\beta)\delta^{2d-\alpha}$ by a simple scaling argument. Then
\begin{align*}
\en_\gamma(E)-\en_\gamma(B_1)
& \leq 2\int_{D_2}\psi(x)\dd x - 2\int_{D_1}\psi(x)\dd x + 2C(d,\alpha,\beta)\delta^{2d-\alpha}\\
& \leq 2 \Bigl(\max_{D_2} \psi - \min_{D_1}\psi\Bigr)|B_\delta| + 2C(d,\alpha,\beta)\delta^{2d-\alpha}\\
& \xupref{proof:L1-min1}{\leq} -2\eta|B_\delta| + 2C(d,\alpha,\beta)\delta^{2d-\alpha}.
\end{align*}
Since $2d-\alpha>d$, we conclude that $\en_\gamma(E)-\en_\gamma(B_1)<0$ if $\delta$ is small enough.

The proof in the case $\psi(R)<\psi(1)$ for some $R>1$ follows a similar construction. Indeed, in this case we can find $\eta>0$ and $\e_1,\e_2>0$ such that
\begin{equation*}
\max_{|\rho-R|\leq\e_1}\psi(\rho) \leq \min_{|\rho-1|\leq \e_2}\psi(\rho) - \eta,
\end{equation*}
and construct the competitor as $E\coloneqq (B_1\setminus B_\delta((1-\delta)e_1))\cup B_\delta(R e_1)$. The details are left to the reader.
\end{proof}

\begin{remark}\label{rmk:L1_min}
In the case $d=3$, $\alpha=1$, $\beta=4$, the computation in \cite{Lop19} shows that the unit ball is a global minimizer for $\en_\gamma$ if and only if $\gamma\geq\gamma_{\rm ball}=\frac{1}{6}$, and that for $\gamma<\gamma_{\rm ball}$ the necessary condition in Proposition~\ref{prop:L1-min} is not satisfied (see Figure~\ref{fig:potential} for a numerical plot of the potential $\psi$ in this case). For these values of the parameters the stability threshold is $\gamma_{*}=\frac{1}{8}<\gamma_{\rm ball}$. Therefore for $\gamma\in(\gamma_*,\gamma_{\rm ball})=(\frac{1}{8},\frac{1}{6})$ the ball is stable but is not an $L^1$-local minimizer. A similar explicit computation can be made in the case $d=3$, $\alpha=1$, $\beta=3$. 
\end{remark}

\begin{figure}
	\begin{center}
		\includegraphics[width=7cm]{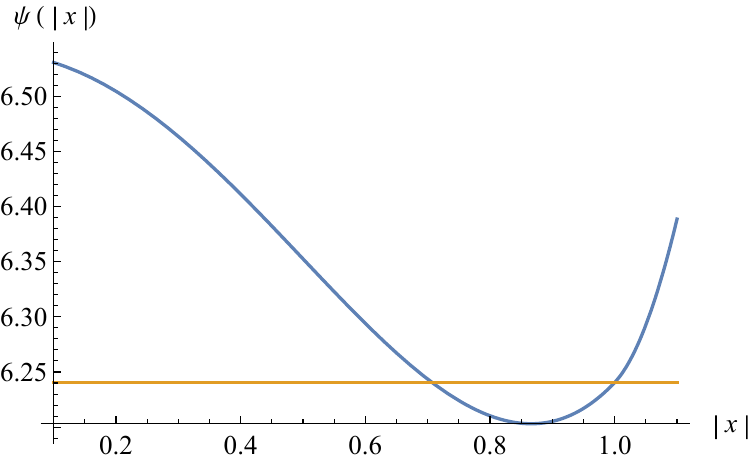}
	\end{center}
	\caption{Numerical plot of the potential $\psi$ of the unit ball, defined in \eqref{eq:potential_ball}, in the case $d=3$, $\alpha=1$, $\beta=4$ and $\gamma=\frac{1}{7}$. Notice that this value of $\gamma$ is larger than the stability threshold $\gamma_{*}=\frac{1}{8}$ (see Remark~\ref{rmk:L1_min}) and that the potential does not satisfy the necessary condition in Proposition~\ref{prop:L1-min}.}
	\label{fig:potential}
\end{figure}


\appendix
\section{Explicit values of the stability thresholds}\label{sec:appendix}

Here we complete the proof of Theorem~\ref{thm:stability} by computing the values of the constants $\gamma_{*}$ and $\gamma_{**}$, defined in \eqref{eq:gammatilde*} and \eqref{eq:gammatilde**} respectively, and showing that the identities \eqref{eq:gamma*} and \eqref{eq:gamma**} hold.

From the definition \eqref{eq:mu} of $\mu_k(\sigma)$ and a straightforward computation, we can write for all $k\geq2$
\begin{equation} \label{app:mu_k}
\frac{\mu_k(-\alpha)-\mu_1(-\alpha)}{\mu_1(\beta)-\mu_k(\beta)} = \kappa(d,\alpha,\beta) X_k,
\end{equation}
where
\begin{equation} \label{app:X_k}
X_k \defeq \frac{1-\displaystyle\prod_{j=1}^{k-1}\frac{\left(j+\frac{\alpha}{2}\right)}{\left(j+\frac{2d-2-\alpha}{2}\right)}}{1-\displaystyle\prod_{j=1}^{k-1}\frac{\left(j-\frac{\beta}{2}\right)}{\left(j+\frac{2d-2+\beta}{2}\right)}}
\end{equation}
and
\begin{equation} \label{app:kappa}
\kappa(d,\alpha,\beta) \defeq 2^{-(\alpha+\beta)} \cdot \frac{\alpha(2d-2+\beta)}{\beta(2d-2-\alpha)} \cdot  \frac{\Gamma\bigl(\frac{d-1-\alpha}{2}\bigr)\Gamma\bigl(\frac{2d-2+\beta}{2}\bigr)}{\Gamma\bigl(\frac{d-1+\beta}{2}\bigr)\Gamma\bigl(\frac{2d-2-\alpha}{2}\bigr)}>0.
\end{equation}
Hence to determine the value of the constants $\gamma_*$ and $\gamma_{**}$ we need to study the behaviour of the sequence $(X_k)_{k\geq2}$, and in particular to compute its supremum and its infimum. The following analysis is similar to the one in \cite[Appendix~C]{FFMMM}.

\begin{figure}
	\begin{center}
		\includegraphics[width=6.5cm]{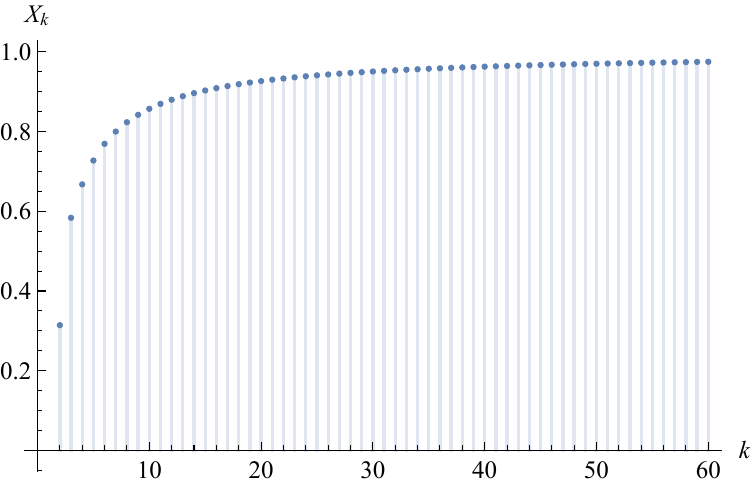}
		\hspace{1cm}
		\includegraphics[width=6.5cm]{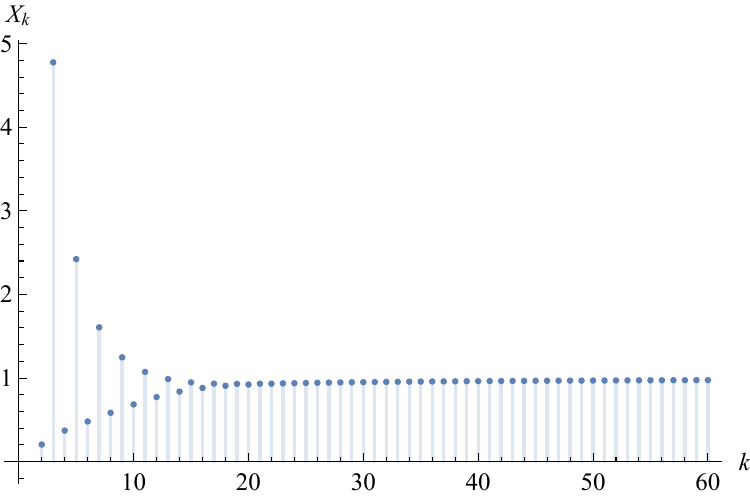}
	\end{center}
	\caption{Numerical plot of the first points of sequence $(X_k)_{k\geq2}$, defined in \eqref{app:X_k}, for two different values of $\beta$, for the choice of parameters $d=3$ and $\alpha=1$ (for which $\beta_*=22$). Left: the case $\beta<\beta_*$ ($\beta=5$). Right: the case $\beta>\beta_*$ ($\beta=155$).}
	\label{fig:sequence}
\end{figure}

\begin{lemma} \label{lem:appendix}
We have that
\begin{equation} \label{eq:sup-inf-X_k}
\inf_{k\geq 2}X_k = X_2, \qquad
\sup_{k\geq 2}X_k=
\begin{cases}
X_3 & \text{if }\beta \geq \beta_*(d,\alpha),\\
\displaystyle\lim_{k\to\infty}X_k=1 & \text{if } \beta < \beta_*(d,\alpha),
\end{cases}
\end{equation}
where $\beta_*(d,\alpha)$ is the constant defined in \eqref{eq:beta*}.
\end{lemma}

\begin{proof}
We start with a few observations. First notice that, thanks to \eqref{eq:limit_prod}, it is immediate to check that $\lim_{k\to\infty}X_k=1$. Following \cite{FFMMM}, we introduce some notation to simplify the computations below. We set
\begin{equation*}
\ell \defeq \frac{d-1}{2}, \qquad t\defeq \ell-\frac{\alpha}{2}, \qquad \tau\defeq \ell+\frac{\beta}{2}
\end{equation*}
(notice that $0<t<\tau$). We also define, for $k\geq2$,
\begin{equation*}
a_k\defeq \prod_{j=1}^{k-1}(j+\ell-t), \quad
b_k\defeq \prod_{j=1}^{k-1}(j+\ell+t), \quad
c_k\defeq \prod_{j=1}^{k-1}(j+\ell-\tau), \quad
d_k\defeq \prod_{j=1}^{k-1}(j+\ell+\tau).
\end{equation*}
With these positions, we have
\begin{equation} \label{app:X_k-2}
X_k = \frac{1-\frac{a_k}{b_k}}{1-\frac{c_k}{d_k}}\,.
\end{equation}
We are now ready to prove the properties in the statement.

\medskip\noindent\textit{Step 1.}
We show that $\inf_{k\geq 2}X_k=X_2$. Suppose first that $0<\beta<2$, and note that in this case $c_k=\prod_{j=1}^{k-1} \left(j-\frac{\beta}{2}\right)>0$ for all $k \geq 2$.

Hence, by a simple calculation, we observe that $X_2<X_k$ is equivalent to
	\[
		(\tld{a}_k \tld{d}_k - \tld{b}_k \tld{c}_k) t \tau + (1+\ell)\big[(\tld{b}_k-\tld{a}_k)\tld{d}_k \tau - (\tld{d}_k-\tld{c}_k)\tld{b}_k t\big]>0,
	\]
where $\tld{a}_k \defeq \frac{a_k}{a_2}$, $\tld{b}_k \defeq \frac{b_k}{b_2}$, $\tld{c}_k \defeq \frac{c_k}{c_2}$, and $\tld{d}_k \defeq \frac{d_k}{d_2}$.
Therefore, to obtain $X_2<X_k$ it suffices to show
	\begin{align}
		\tld{a}_k \tld{d}_k - \tld{b}_k \tld{c}_k &> 0 \qquad \text{for all } k \geq 3,   \label{app:X2-first-cond} \\
		(\tld{b}_k-\tld{a}_k)\tld{d}_k \tau - (\tld{d}_k-\tld{c}_k)\tld{b}_k t &>0 \qquad \text{for all } k \geq 3. \label{app:X2-second-cond}
	\end{align}
	
To prove \eqref{app:X2-first-cond}, note that
	\[
		(j+\ell-t)(j+\ell+\tau) - (j+\ell+t)(j+\ell-\tau) = 2(j+\ell)(\tau-t)>0.
	\]
Since $\beta<2$ by assumption, $j+\ell-\tau>0$ for all $j=2,\ldots,k-1$; hence,
	\[
		\tld{a}_k \tld{d}_k = \prod_{j=2}^{k-1} (j+\ell-t)(j+\ell+\tau) > \prod_{j=2}^{k-1} (j+\ell+t)(j+\ell-\tau) = \tld{b}_k \tld{c}_k.
	\]
	
We prove \eqref{app:X2-second-cond} by induction. For $k=3$ an explicit calculation yields
	\[
		(\tld{b}_3-\tld{a}_3)\tld{d}_3 \tau - (\tld{d}_3-\tld{c}_3)\tld{b}_3 t = 2 t \tau (\tau-t) >0.
	\]
Now assume that \eqref{app:X2-second-cond} holds for some $k$. Then, by a straightforward computation and by noting that $t<\tau$, we have
	\begin{align*}
		(\tld{b}_{k+1}-\tld{a}_{k+1})\tld{d}_{k+1} \tau - (\tld{d}_{k+1}-\tld{c}_{k+1})\tld{b}_{k+1} t  &= 2t\tau \tld{b}_k \tld{d}_k(k+\ell+\tau)-2t\tau \tld{b}_k \tld{d}_k(k+\ell+t)\\
																		&\qquad\qquad +(k+\ell-t)(k+\ell+\tau)(\tld{b}_k-\tld{a}_k)\tld{d}_k \tau \\
																		&\qquad\qquad\qquad -(k+\ell-\tau)(k+\ell+t)(\tld{d}_k-\tld{c}_k)\tld{b}_k t \\
																		&	\xupref{app:X2-second-cond}{\geq} 2(k+\ell)(\tau-t)(\tld{d}_k-\tld{c}_k)\tld{b}_k t > 0,
	\end{align*}
and \eqref{app:X2-second-cond} follows by induction. This completes the proof of the claim in the case $0<\beta<2$.

Now suppose $\beta \geq 2$, and consider
	\[
		\frac{X_k}{X_2} = \frac{1-\frac{a_k}{b_k}}{1-\frac{a_2}{b_2}} \ \frac{1-\frac{c_2}{d_2}}{1-\frac{c_k}{d_k}}\,.
	\]
Note that
	\[
		\prod_{j=2}^{k-1} (j+\ell+t) > \prod_{j=2}^{k-1} (j+\ell-t)  \qquad\Longleftrightarrow\qquad \frac{1-\frac{a_k}{b_k}}{1-\frac{a_2}{b_2}} > 1,
	\]
and the inequality on the left-hand side hold trivially since all factors are positive. On the other hand,
	\[
		\frac{1+\ell-\tau}{1+\ell+\tau} \leq \prod_{j=1}^{k-1} \frac{j+\ell-\tau}{j+\ell+\tau} \qquad\Longleftrightarrow\qquad \frac{1-\frac{c_2}{d_2}}{1-\frac{c_k}{d_k}} \geq 1.
	\]
The inequality on the left-hand side holds since we have that $\prod_{j=2}^{k-1} \frac{j+\ell-\tau}{j+\ell+\tau} \leq 1$, and $\tau \geq 1+\ell$ thanks to the assumption $\beta \geq 2$. Therefore $X_k > X_2$ for any $k \geq 3$.

\medskip\noindent\textit{Step 2.}
To complete the proof it remains to determine the supremum of the sequence $(X_k)_{k\geq2}$. We preliminary observe that, by \eqref{app:X_k-2}, the condition $X_k<1$ is equivalent to
\begin{equation*}
\frac{(b_k-a_k)d_k-(d_k-c_k)b_k}{(d_k-c_k)b_k}<0.
\end{equation*}
Since $b_k>0$ and $d_k-c_k>0$ for every $k$, as it is easy to check, we obtain the equivalence of the two conditions
\begin{equation} \label{app:X_k<1}
X_k < 1 \qquad\Longleftrightarrow\qquad a_kd_k - b_kc_k>0.
\end{equation}
In particular, we can compare the value of $X_3$ with $1$, which is the limit value of the sequence as $k\to\infty$. By a straightforward calculation one can check that
\begin{equation} \label{app:X_3<1}
X_3 < 1 \quad\Longleftrightarrow\quad a_3d_3 - b_3c_3>0 \quad\Longleftrightarrow\quad 2+3\ell+\ell^2-t\tau>0 \quad\Longleftrightarrow\quad \beta<\beta_*
\end{equation}
where $\beta_*=\beta_*(d,\alpha)$ is the constant defined in \eqref{eq:beta*}. This computation also explains the role of the constant $\beta_*$: it is the unique value of $\beta$, for given $d$ and $\alpha$, such that $X_3=1=\lim_{k\to\infty}X_k$.

We first consider the case $\beta<\beta_*$ and we prove that in this case the supremum of the sequence $(X_k)_k$ is exactly 1, that is the limit value as $k\to\infty$. More precisely, we claim that
\begin{equation}\label{app:claim1}
\text{if }\beta<\beta_*, \text{ then }X_k<1\text{ for all }k\geq3.
\end{equation}
Recalling \eqref{app:X_k<1}, we will obtain the proof of the claim by showing that
\begin{equation} \label{app:induction-claim1}
|b_kc_k|\leq a_kd_k \qquad\text{for all }k\geq3.
\end{equation}
We prove this inequality by induction. For $k=3$ we have $b_3c_3<a_3d_3$ since we are assuming $\beta<\beta_*$ (recall \eqref{app:X_3<1}). On the other hand, we need to show that $b_3c_3+a_3d_3>0$. We have by definition
\begin{align*}
b_3c_3+a_3d_3
& = (1+\ell+t)(2+\ell+t)(1-\textstyle\frac{\beta}{2})(2-\frac{\beta}{2}) + (1+\frac{\alpha}{2})(2+\frac{\alpha}{2})(1+\ell+\tau)(2+\ell+\tau).
\end{align*}
From this expression, it is clear that if $\beta\leq2$ or $\beta\geq4$ both terms are positive (or zero) and we obtain $b_3c_3+a_3d_3>0$, as desired. If $2< \beta< 4$, the first term is negative, however the sum of the two terms is positive, as
\begin{align*}
0< (1+\ell+t)(2+\ell+t)<(1+\ell+\tau)(2+\ell+\tau), \quad |(1-\textstyle\frac{\beta}{2})(2-\frac{\beta}{2})|  < 1 < (1+\frac{\alpha}{2})(2+\frac{\alpha}{2}).
\end{align*}

We next prove the induction step. Suppose that the inequality \eqref{app:induction-claim1} holds for some $k\geq3$, and let us prove it for $k+1$. We have
\begin{align*}
b_{k+1}c_{k+1} = b_kc_k(k+\ell+t)(k+\ell-\tau), \qquad a_{k+1}d_{k+1} = a_kd_k(k+\ell-t)(k+\ell+\tau).
\end{align*}
By the inductive step $|b_kc_k|<|a_kd_k|$, therefore in view of the previous identities it is enough to show that
$$
|(k+\ell+t)(k+\ell-\tau)| < (k+\ell-t)(k+\ell+\tau).
$$
This can be checked directly: indeed $(k+\ell+t)(k+\ell-\tau)-(k+\ell-t)(k+\ell+\tau)=-2(k+\ell)(\tau-t)<0$, and on the other hand $(k+\ell+t)(k+\ell-\tau)+(k+\ell-t)(k+\ell+\tau)=2(k^2+2k\ell+\ell^2-t\tau)>2(3+2\ell+\ell^2-t\tau)>0$ by \eqref{app:X_3<1}.
This completes the proof of \eqref{app:claim1}.

\medskip\noindent\textit{Step 3.}
We finally determine the supremum of the sequence $(X_k)_k$ in the case $\beta\geq\beta_*$, and we show that it is attained for $k=3$: we claim that
\begin{equation}\label{app:claim2}
\text{if }\beta\geq\beta_*, \text{ then }X_k \leq X_3\text{ for all }k\geq4.
\end{equation}

We first prove \eqref{app:claim2} for $k$ even, by showing that $X_k\leq 1$ for all $k$ even. By \eqref{app:X_k<1} we have that $X_k\leq 1$ if and only if $\frac{c_k}{d_k}\leq \frac{a_k}{b_k}$, which can be written as $h_k(\tau)\leq h_k(t)$ for $h_k(x)\coloneqq\prod_{j=1}^{k-1}\frac{j+\ell-x}{j+\ell+x}$. Since $t\in(0,\ell)$ and $\tau\in(\ell,\infty)$, we obtain that $X_k \leq 1$ for $k$ even, if we show that
\begin{equation} \label{app:even-1}
\sup_{x\in(\ell,\infty)}h_k(x) \leq \inf_{x\in(0,\ell)}h_k(x) \qquad\text{for all $k$ even.}
\end{equation}
By differentiating it is easy to see that $h_k$ is decreasing in $(0,\ell+1)$. Moreover, $h_k(x)<0$ for all $x>\ell+k-1$ and $k$ even, since in this case $h_k(x)$ is the product of an odd number of negative terms. It follows by the previous properties that $h_k(x)<h_k(y)$ for all $x\in(\ell,\ell+1]\cup[\ell+k-1,\infty)$ and $y\in(0,\ell)$. Then to complete the proof of \eqref{app:even-1}, bearing in mind the monotonicity of $h_k$, it is enough to show that $h_k(\ell)>h_k(x)$ for all $x\in(\ell+1,\ell+k-1)$, which is equivalent to
\begin{equation} \label{app:even-2}
\prod_{j=1}^{k-1}\frac{j}{j+2\ell} > \prod_{j=1}^{k-1}\frac{j+\ell-x}{j+\ell+x} \qquad\text{for all }x\in(\ell+1,\ell+k-1).
\end{equation}
Let $r\in\{1,\ldots,k-2\}$ be such that $x\in[\ell+r,\ell+r+1)$: then by reordering the terms on the left-hand side we can write \eqref{app:even-2} as
\begin{equation*}
\Biggl(\prod_{j=1}^{r}\frac{r-j+1}{r-j+1+2\ell}\Biggr)\Biggl(\prod_{j=r+1}^{k-1}\frac{j}{j+2\ell}\Biggr) >
\Biggl(\prod_{j=1}^{r}\frac{j+\ell-x}{j+\ell+x}\Biggr)\Biggl(\prod_{j=r+1}^{k-1}\frac{j+\ell-x}{j+\ell+x}\Biggr).
\end{equation*}
By the condition $x\in[\ell+r,\ell+r+1)$ it is now easily checked that $|\frac{j+\ell-x}{j+\ell+x}|\leq\frac{r-j+1}{r-j+1+2\ell}$ for all $j\in\{1,\ldots,r\}$, and $|\frac{j+\ell-x}{j+\ell+x}|\leq\frac{j}{j+2\ell}$ for all $j\in\{r+1,\ldots,k-1\}$. Hence every term in the product on the right-hand side is, in absolute value, smaller than the corresponding term in the product on the left-hand side. This proves \eqref{app:even-2} and completes the proof of the fact that $X_k\leq 1$ for all $k\geq4$ even. Since $X_3\geq1$ (by the assumption $\beta\geq\beta_*$) this proves \eqref{app:claim2} in the case $k$ even.

Eventually, we prove \eqref{app:claim2} in the case $k$ odd. By \eqref{app:X_k-2} and a simple calculation, the condition $X_k\leq X_3$ is equivalent to
\begin{equation*}
\frac{c_k}{d_k}d_3 t - \frac{a_k}{b_k}b_3\tau + (2+3\ell+\ell^2-t\tau)(\tau-t) \leq 0.
\end{equation*}
Letting $\sigma_k\defeq \frac{c_k}{d_k}d_3 t - \frac{a_k}{b_k}b_3\tau$ and $\eta\defeq-(2+3\ell+\ell^2-t\tau)(\tau-t)$, we need to show that 
\begin{equation} \label{app:odd-1}
\sigma_k \leq \eta \qquad \text{for all $k\geq3$ odd}. 
\end{equation}
Notice  that $\sigma_3=\eta$ and $\lim_{k\to\infty}\sigma_k\leq\eta$, since the assumption $\beta\geq\beta_*$ gives $\lim_{k\to\infty}X_k\leq X_3$. In view of these two properties, to prove \eqref{app:odd-1} it is enough to show the following claim:
\begin{equation} \label{app:odd-2}
\text{if } \sigma_{k+2}>\sigma_k \text{ for some $k\geq3$ odd, then } \sigma_{k+4}>\sigma_{k+2}.
\end{equation}
By a straightforward computation one has
\begin{equation} \label{app:odd-3}
\sigma_{k+2} > \sigma_{k} \qquad\Longleftrightarrow\qquad \prod_{j=1}^{k-1}\frac{j+\ell-t}{j+2+\ell+t} > \prod_{j=1}^{k-1}\frac{j+\ell-\tau}{j+2+\ell+\tau}\,.
\end{equation}

Assume now that $\sigma_{k+2}>\sigma_{k}$ for some $k\geq3$ odd. We distinguish three cases depending on the value of $k$. If $k>\tau-\ell$, then
\begin{equation} \label{app:odd-4}
\frac{(k+\ell-t)(k+1+\ell-t)}{(k+2+\ell+t)(k+3+\ell+t)} > \frac{(k+\ell-\tau)(k+1+\ell-\tau)}{(k+2+\ell+\tau)(k+3+\ell+\tau)}
\end{equation}
(since $\tau>t$ and the factors on the right-hand side are positive). Then by the assumption $\sigma_{k+2}>\sigma_{k}$, \eqref{app:odd-3}, and \eqref{app:odd-4}, we obtain
\begin{equation} \label{app:odd-5}
\prod_{j=1}^{k+1}\frac{j+\ell-t}{j+2+\ell+t} > \prod_{j=1}^{k+1}\frac{j+\ell-\tau}{j+2+\ell+\tau}\,,
\end{equation}
which is equivalent to $\sigma_{k+4}>\sigma_{k+2}$, as desired.

Consider next the case $\tau-\ell-1\leq k \leq \tau-\ell$. In this case the inequality \eqref{app:odd-5} is trivially satisfied, since the left-hand side is positive and the right-hand side is negative or zero (as $k$ is odd). Hence also in this case $\sigma_{k+4}>\sigma_{k+2}$.

Finally, let $k<\tau-\ell-1$. We claim that
\begin{equation} \label{app:odd-6}
0<\frac{\tau-k-\ell}{k+2+\ell+\tau}<\frac{k+\ell-t}{k+2+\ell+t}\,, \qquad 0<\frac{\tau-(k+1)-\ell}{k+3+\ell+\tau}<\frac{k+1+\ell-t}{k+3+\ell+t}\,.
\end{equation}
Indeed, first notice that
$$
\frac{\tau-j-\ell}{j+2+\ell+\tau}\geq\frac{j+\ell-t}{j+2+\ell+t} \qquad\Longleftrightarrow\qquad t+\tau+t\tau \geq j^2+2j+2j\ell+2\ell+\ell^2.
$$
If the first inequality in \eqref{app:odd-6} fails, then $t+\tau+t\tau \geq k^2+2k+2k\ell+2\ell+\ell^2> j^2+2j+2j\ell+2\ell+\ell^2$ for all $j\in\{1,\ldots,k-1\}$, hence
$$
\frac{\tau-j-\ell}{j+2+\ell+\tau} > \frac{j+\ell-t}{j+2+\ell+t} \qquad\text{for all }j\in\{1,\ldots,k-1\}.
$$
In view of \eqref{app:odd-3} this contradicts the assumption $\sigma_{k+2}>\sigma_{k}$. Similarly, also the second inequality in \eqref{app:odd-6} must be true. By \eqref{app:odd-6} we deduce that the inequality \eqref{app:odd-4} holds also in this case and, arguing as before, we obtain \eqref{app:odd-5}, which is equivalent to $\sigma_{k+4}>\sigma_{k+2}$.

Therefore we proved the implication \eqref{app:odd-2} in all cases, which yields \eqref{app:odd-1} and, in turn, the conclusion of the lemma.
\end{proof}

\begin{corollary} \label{cor:appendix}
The identities \eqref{eq:gamma*} and \eqref{eq:gamma**} hold.
\end{corollary}

\begin{proof}
By the definition \eqref{eq:gammatilde*}--\eqref{eq:gammatilde**} of $\gamma_*$ and $\gamma_{**}$, the identity \eqref{app:mu_k}, and Lemma~\ref{lem:appendix}, we have that
\begin{equation} \label{eq:sup-inf-mu_k}
\gamma_{*} =
\begin{cases}
\kappa(d,\alpha,\beta) X_3 & \text{if }\beta \geq \beta_*(d,\alpha),\\
\kappa(d,\alpha,\beta) & \text{if } \beta < \beta_*(d,\alpha),
\end{cases}
\qquad
\gamma_{**} = \kappa(d,\alpha,\beta)X_2.
\end{equation}
Therefore to conclude the proof it is enough to check that the values in \eqref{eq:sup-inf-mu_k} coincide with those in \eqref{eq:gamma*} and \eqref{eq:gamma**}.

We first express $\kappa(d,\alpha,\beta)$ in terms of the energies $\nl_{-\alpha}(B_1)$ and $\nl_{\beta}(B_1)$. Indeed, by using \eqref{eq:energyball} we compute
\begin{equation*}
\frac{\nl_{-\alpha}(B_1)}{\nl_{\beta}(B_1)} =
2^{-(\alpha+\beta)} \frac{(d+\beta)(2d+\beta)(d-1-\alpha)(2d-2+\beta)}{(d-\alpha)(2d-\alpha)(d-1+\beta)(2d-2-\alpha)} \, \frac{\Gamma\bigl(\frac{d-1-\alpha}{2}\bigr)\Gamma\bigl(\frac{2d-2+\beta}{2}\bigr)}{\Gamma\bigl(\frac{d-1+\beta}{2}\bigr)\Gamma\bigl(\frac{2d-2-\alpha}{2}\bigr)} \,.
\end{equation*}
By inserting this expression into \eqref{app:kappa} we find
\begin{equation} \label{app:kappa-2}
\kappa(d,\alpha,\beta) = \frac{\alpha(d-\alpha)(2d-\alpha)(d-1+\beta)}{\beta(d+\beta)(2d+\beta)(d-1-\alpha)}\cdot \frac{\nl_{-\alpha}(B_1)}{\nl_{\beta}(B_1)}.
\end{equation}
Moreover, it is straightforward to compute
\begin{equation}\label{app:X_2-X_3}
X_2 = \frac{d-1-\alpha}{d-1+\beta}\,, \qquad\qquad
X_3 = \frac{(d-1-\alpha)(2d+\beta)(2d+2+\beta)}{(d-1+\beta)(2d-\alpha)(2d+2-\alpha)}\,.
\end{equation}
By inserting \eqref{app:kappa-2} and \eqref{app:X_2-X_3} into \eqref{eq:sup-inf-mu_k} we obtain the explicit values of $\gamma_*$ and $\gamma_{**}$.
\end{proof}


\bigskip
\subsection*{Acknowledgments}

The authors would like to thank Rupert Frank for his valuable comments on the preprint version of this paper. IT’s research is partially supported by the Simons Collaboration Grant for Mathematicians No. 851065.
MB is member of the GNAMPA group of INdAM.

\bibliographystyle{IEEEtranSA}
\def\url#1{}
\bibliography{references}

\end{document}